\documentclass[a4paper,10pt,english,authoryear,leqno]{amsart}

\usepackage{babel}
\usepackage{latexsym,textcomp,amssymb}
\usepackage{mathpazo,pifont}
\usepackage[only,lightning]{stmaryrd}
\usepackage{amsmath,calc,amsthm}
\usepackage{natbib}
\usepackage[final,protrusion=true,expansion=false]{microtype}

\newtheorem{thm}{Theorem}[section]
\newtheorem{lemma}[thm]{Lemma}
\newtheorem{cor}[thm]{Corollary}

\newtheorem{exs}[thm]{Examples}

\theoremstyle{definition}
\newtheorem{defn}[thm]{Definition}
\newtheorem{rem}[thm]{Remark}

\newcommand{\set}[2]{\{#1\mid#2\}}
\newcommand{\erz}[1]{\langle#1\rangle}
\newcommand{\anzahl}[1]{\vphantom{#1}^{\#}#1}

\newcommand{\bm}{\partial}

\newcommand{\quotient}[2]{#1/#2}
\newcommand{\siehe}{cf.~}

\let\ker\undefined
\DeclareMathOperator{\im}{im}
\DeclareMathOperator{\ker}{ker}
\DeclareMathOperator{\Dim}{dim}

\DeclareMathOperator{\sk}{sk}
\DeclareMathOperator{\supp}{supp}
\DeclareMathOperator{\bd}{bd}
\DeclareMathOperator{\degf}{dgf}
\DeclareMathOperator{\gen}{gen}

\newcommand{\contradiction}{\ensuremath{\lightning}}

\newcommand{\N}{\mathbb{N}}
\newcommand{\Z}{\mathbb{Z}}

\newcommand{\ie}{i.\,e.\ }

\title{Shellability and Regularity of Chain Complexes over a Principal Ring}
\date{\today}

\author{Gerrit Grenzebach}
\author{Bj\"orn Walker}

\address{Fachbereich Mathematik, Universit\"at Bremen, Bibliothekstr.~1, 28359 Bremen, Germany}
\address{Gymnasium Carolinum, Gro\ss{}e Domsfreiheit 1, 49074 Osnabr\"uck, Germany}


\begin{document}
\pagenumbering{Roman}

\begin{abstract}
  The goal of this paper is to generalize some of the existing toolkit of 
  combinatorial algebraic topology in order to study the homology of abstract chain complexes.
  We define \emph{shellability of chain complexes} in a similar way as for cell complexes
  and introduce the notion of \emph{regular} chain complexes. 
  
  In the case of chain complexes coming from simplicial 
  complexes we recover the classical notions but, in contrast to the topological 
  case, in the abstract setting shellings turn out to be a weaker homological 
  invariant. In particular, we study special chain complexes which are \emph{cones}, 
  and a class of regular chain complexes, which we call \emph{totally regular}, for which 
  we can obtain complete homological information.  
\end{abstract}

\maketitle

\pagenumbering{arabic}
\pagestyle{headings}

\section{Introduction}
This paper aims to generalize the notion of shellability 
to abstract chain complexes. 
A simplicial complex is shellable if there is an order of its maximal simplices
$F_{1}$, $F_{2}$, \ldots, $F_{t}$ such that 
$(\bigcup_{i=1}^{k-1}F_{i})\cap F_{k}$
is a pure simplicial complex of dimension 
$\Dim F_{k}-1$ for each $1\leq k\leq t$ \citep[cf.][page~211]{Kozlov.2008a}.
Now, any simplicial complex gives rise to a chain complex 
\citep[cf.][page~104]{Hatcher.2008}, therefore one is naturally led to ask 
for an algebraic counterpart of shellability. 
As far as we know, this question has never been studied before.
As a point of interest, it turns out that 
the algebraic situation is more complicated than the combinatorial case.  
In the end, so called totally regular chain complexes (a special kind of regular ones) 
turn out to be an analogue to shellable simplicial complexes as they have the same homology. 

In Section~\ref{ssec:principalrings} we shortly present some basic facts
on principal rings and free modules before we introduce abstract chain complexes
in Section~\ref{sec:chain}. Thereby, we define critical basis elements as an analogue 
to spanning simplices in simplicial complexes \citep[cf.][page~212]{Kozlov.2008a}
and compute the homology of pure chain complexes having critical basis elements. 
Section~\ref{sec:acyclic} contains a small excursion about acyclic chain complexes 
and cones. 

Both the last sections deal with our main topic: shellable and regular chain 
complexes. After defining shellability for chain complexes in Section~\ref{sec:shell},
we prove the existence of a special shelling which we call monotonically descending.
We also show that $i$-skeletons of shellable chain complexes are shellable themselves. 
In contrast to shellable simplicial complexes, the homology of shellable chain complexes 
is not known in general. Therefore, we introduce regular chain complexes as a special 
class of shellable chain complexes in Section~\ref{sec:regular}. As above, any $i$-skeleton
of a regular chain complex is also regular itself. 
Finally, we compute the homology of special regular chain complexes 
which are called totally regular. 

\section{Preliminaries: Principal Rings and Free Modules}
\label{ssec:principalrings}

A \emph{principal ring} is a commutative ring with 1 which is an 
integral domain whose ideals are all principal \citep[cf.][page~35]{bosch}. 
If $M$ is a free module over a principal 
ring $R$, then every submodule of $M$ is free \citep[cf.][page~146]{lang}.

\medskip
We will mostly use finitely generated modules, \ie modules of the 
form $M=\bigoplus_{i=1}^{n}Re_{i}$. For these modules we define:
\begin{itemize}
	\item  The \emph{generating number} $\gen_{R}M$ is the minimal number of 
	elements which generate $M$ 
	\citep[cf.][page~90]{OeljeklausRemmert.1974}:
	\begin{equation*}
		\gen_{R}M:=\min\set{n\in\N}{M\text{ is generated by $n$ elements.}}
	\end{equation*}

	\item  The \emph{degree of freedom} $\degf_{R}M$ is the maximal 
	number of linearly independent elements in $M$
	\citep[cf.][page~110]{OeljeklausRemmert.1974}:
	\begin{equation*}
		\degf_{R}M:=\max\set{n\in\N}{\text{There are $n$ independent elements 
		in $M$.}}
	\end{equation*}	
\end{itemize}

The next two theorems are quite helpful and we will use them later 
in Section~\ref{ssec:critical}
\citep[cf.][pages~112~and~115]{OeljeklausRemmert.1974}.
\begin{thm}
	\label{thm:degree}
	Let $R$ be an integral domain. 
	Let $M$ and $N$ be finitely generated $R$-modules. 
	For any $R$-linear mapping $\varphi\colon M\rightarrow N$ holds:
	\begin{equation*}
		\degf_{R}(M)=\degf_{R}(\ker\varphi) + \degf_{R}(\im\varphi). 
	\end{equation*}
\end{thm}

\begin{thm}
	\label{thm:genequaldegree}
	Let $R$ be an integral domain. 
	If $M$ is a finitely generated free $R$-module, then
	$\gen_{R}M=\degf_{R}M$. 
\end{thm}

\section{About Chain Complexes}
\label{sec:chain}
In the following let $R$ always be a principal ring. We want to introduce
now the basic concepts of chain modules over a principal ring before we 
define an analogue to a \emph{spanning simplex} 
\citep[cf.][page~212]{Kozlov.2008a}.

\subsection{First Concepts}
\label{ssec:concepts}
\begin{defn}
	For every $\nu\in\N$ let $C_{\nu}$ be 
	a free module over $R$ with basis 
	$\Omega_{\nu}:=\{e^{\nu}_{1},\ldots,e^{\nu}_{k_{\nu}}\}$, 
	$k_{\nu}\geq 1$,
	or $\Omega_{\nu}:=\emptyset$.
	Let
	\begin{align*}
		\bm_{\nu} & \colon C_{\nu}\rightarrow C_{\nu -1}, \quad 
		\nu\geq 1,  \\
		\bm_{0} & \colon C_{0}\rightarrow 0
	\end{align*}
	be $R$-linear mappings 
	so that $\bm_{\nu}\circ\bm_{\nu +1}=0$ 
	(\ie $\im\bm_{\nu +1}\subset\ker\bm_{\nu}$). 
	Then we call 
	\begin{equation*}
		\ldots\rightarrow C_{i+1} \stackrel{\bm_{i+1}}{\longrightarrow}
		C_{i} \rightarrow \ldots
		\rightarrow 
		C_{1} \stackrel{\bm_{1}}{\longrightarrow}
		C_{0}\stackrel{\bm_{0}}{\longrightarrow} 0
	\end{equation*}
	 a \emph{chain complex} $(C, \Omega)$ with basis 
	$\Omega := \bigcup_{\nu\in\N}\Omega_{\nu}$. The free modules 
	$C_{\nu}$ are called $\emph{chain modules}$, 
	and the mappings $\bm_{\nu}$ are \emph{boundary mappings}.
	
	A chain complex $(C, \Omega)$ is \emph{finite of order} 
	$d$ if $C_{\nu}=0$ for all $\nu > d$ and $C_{d}\neq0$. 
\end{defn}

\begin{rem}
  We always consider chain complexes $C$ together with a fixed basis 
  $\Omega=\bigcup_{\nu\in\N}\Omega_{\nu}$ and we do not want to change 
  the basis of any chain complex except for permutations of the basis elements.   
\end{rem}

\begin{defn}
	\label{defn:skeleton}
	Let $(C, \Omega)$ be a chain complex. For $0\leq i$, 
	the \emph{$i$-skeleton} $\sk_{i}(C)$ is a 
	finite subcomplex of $C$ of order $i$
	whose chain modules are
	$\bigl(\sk_{i}(C) \bigr)_{\nu} = C_{\nu}$ for 
	$0\leq\nu\leq i$.
\end{defn}
So a basis of $\sk_{i}(C)$ is $\bigcup\limits_{\nu=0}^{i}\Omega_{\nu}$. 

In particular, if $(C,\Omega)$ is a finite chain complex of order $d$, then the 
d-skeleton $\sk_{d}(C)$ of $C$ is equal to $C$: $\sk_{d}(C)=C$. 

\begin{defn}
	Let $C_{\nu}$ be a chain module with basis 
	$\Omega_{\nu}=\{e^{\nu}_{1}, \ldots, e^{\nu}_{k_{\nu}}\}$ 
	in a chain complex $(C,\Omega)$. 
	For an element $x=\sum_{i=1}^{k_{\nu}}a_{i}e^{\nu}_{i}$, the 
	\emph{support} of $x$ is the set of all basis elements 
	$e^{\nu}_{i}$ with coefficient $a_{i}\neq 0$:
	\begin{equation*}
		\supp(x):=\set{e^{\nu}_{i}}{a_{i}\neq 0}\subset\Omega_{\nu}.
	\end{equation*}
	
	The \emph{boundary} of $x$ is the support of $\bm_{\nu}(x)$:
	\begin{equation*}
		\bd(x):=\supp(\bm_{\nu}x).
	\end{equation*}
\end{defn}

Notice the two facts:
\begin{align*}
	\supp(x)=\emptyset & \iff x=0;  \\
	\bd(x)=\emptyset & \iff x\in\ker(\bm_{\nu}).
\end{align*}

\begin{defn}
	A finite chain complex $(C, \Omega)$ of order $d$ is \emph{pure} if any 
	$e^{\nu}_{j}\in\Omega$ for $0\leq\nu\leq d-1$ and $0\leq j\leq k_{\nu}$ 
	is contained in the boundary 
	of some $e^{\nu+1}_{\ell}$, $1\leq \ell\leq k_{\nu+1}$, \ie
	\begin{equation*}
		e^{\nu}_{j}\in\bd(e^{\nu+1}_{\ell})
		=\supp (\bm_{\nu+1}e^{\nu+1}_{\ell}).
	\end{equation*}
\end{defn}

\begin{defn}
	A basis element $e\in\Omega$ of a chain complex $(C,\Omega)$ is called \emph{maximal}
	if it is not contained in the boundary of any other basis element. 
\end{defn}

\begin{rem}
	If $(C,\Omega)$ is a finite chain complex of order $d$ then
	all basis elements of $\Omega_{d}$ are maximal. Furthermore, 
	if $C$ is even a pure chain complex, the maximal basis elements 
	are exactly the ones in $\Omega_{d}$. 
\end{rem}

Now we introduce an analogue to a simplex in a simplicial 
complex:
\begin{defn}
	Let $(C, \Omega)$ be a chain complex with basis $\Omega$. For a
	basis element $e^{\mu}_{j}\in\Omega_{\mu}$ we denote the subcomplex
	of $C$ with $(C_{e^{\mu}_{j}},\Omega_{e^{\mu}_{j}})$ whose chain
	modules $\bigl(C_{e^{\mu}_{j}}\bigr)_{\nu}$ have the following
	bases $\bigl(\Omega_{e^{\mu}_{j}}\bigr)_{\nu}$:
	\begin{itemize}
		\item  $\bigl(\Omega_{e^{\mu}_{j}}\bigr)_{\nu}=\emptyset$
		\quad for $\nu\geq\mu+1$;
	
		\item  $\bigl(\Omega_{e^{\mu}_{j}}\bigr)_{\mu}=\{e^{\mu}_{j}\}$;
 		
		\item  $\bigl(\Omega_{e^{\mu}_{j}}\bigr)_{\nu}
		=\bigcup\limits_{e\in\bigl(\Omega_{e^{\mu}_{j}}\bigr)_{\nu+1}}\bd(e)$
		\quad for $0\leq\nu\leq\mu -1$. 
	\end{itemize}
\end{defn}

\begin{rem}
	The subcomplex $C_{e^{\mu}_{j}}$ is a pure finite chain complex of
	order $\mu$. In particular, 
	$\bigl(\Omega_{e^{\mu}_{j}}\bigr)_{\mu -1}=\bd(e^{\mu}_{j})$.
	
	If there is a basis element $e^{\lambda}_{i}\in\Omega_{\lambda}$ 
	so that
	$e^{\lambda}_{i}\in\Omega_{e^{\mu}_{j}}\cap\Omega_{e^{\kappa}_{\ell}}$ 
	for $\lambda < \mu, \kappa$, then 
	$\Omega_{e^{\lambda}_{i}}\subset
	\Omega_{e^{\mu}_{j}}\cap\Omega_{e^{\kappa}_{\ell}}$, \ie the 
	complex $C_{e^{\lambda}_{i}}$ is contained in the chain complex 
	generated by $\Omega_{e^{\mu}_{j}}\cap\Omega_{e^{\kappa}_{\ell}}$.
\end{rem}

\subsection{Critical Basis Elements in Chain Complexes}
\label{ssec:critical}
\begin{defn}
	\label{defn:critical}
	Let $(C,\Omega)$ be a chain complex over R. Let
	$\Gamma$ be the set of all maximal basis elements of $C$: 
	\begin{equation*}
		\Gamma := 
		\set{e\in\Omega}{e\not\in\bd(f) \text{ for all }f\in\Omega}.
	\end{equation*}
	Let each 
	basis	$\Omega_{\nu}$ of a chain module $C_{\nu}$ have an 
	ordering in which the elements of $\Omega_{\nu}\setminus\Gamma$ 
	come first. 
	A basis element $e^{\nu}_{j}\in(\Omega_{\nu}\cap\Gamma)$ with $j\geq 
	2$ is called:
	\begin{itemize}
		\item  \emph{critical} if there exist $a_{i}\in R$, 
		$1\leq i\leq j-1$, so that
		\begin{equation*}
			\bm_{\nu}(e^{\nu}_{j})
			=\sum_{i=1}^{j-1}a_{i}\bm_{\nu}(e^{\nu}_{i});
		\end{equation*}
	
		\item  \emph{precritical} if there exist $a_{i}\in R$, 
		$1\leq i\leq j$, $a_{j}\neq 0$, so that
		\begin{equation*}
			a_{j}\bm_{\nu}(e^{\nu}_{j})
			=\sum_{i=1}^{j-1}a_{i}\bm_{\nu}(e^{\nu}_{i});
		\end{equation*}
		
		\item  \emph{noncritical} if $e^{\nu}_{j}$ is neither 
		critical nor precritical. 
	\end{itemize}
\end{defn}

\begin{rem}
	\begin{enumerate}
		\item  A critical basis element corresponds to a 
		\emph{spanning simplex} in a simplicial complex 
		\citep[cf.][page~212]{Kozlov.2008a}.
	
		\item A critical basis element is always precritical.
		Conversely, a precritical element $e^{\nu}_{j}$ is critical if
		the coefficient $a_{j}$ is a unit in $R$. Hence, the terms
		\emph{precritical} and \emph{critical} coincide if the principal ring
		$R$ is a field.
	
		\item  $\supp\bigl(a_{j}\bm_{\nu}(e^{\nu}_{j})\bigr)
		=\supp\bigl(\bm_{\nu}(e^{\nu}_{j})\bigr)$ if $a_{j}\neq 0$.
		
		\item  It is possible to change the ordering of a basis $\Omega_{\nu}$ so 
		that all noncritical basis elements come first. 
	\end{enumerate}
\end{rem}

In a pure chain complex of order $d$, the precritical elements in the 
chain module basis $\Omega_{d}$ can be seen as the generators of homology. 
If all  basis elements of $\Omega_{d}$ are either 
noncritical or critical, we can name a basis of $H_{d}(C)$.  
We follow Bj\"orner \cite[page~254]{Bjorner.1992} who has done this 
for the special case of simplicial complexes and generalize his proof 
to chain complexes. 

To formulate a theorem for $d\geq 1$ we introduce a new notation. 
Let 
$C_{\nu}$ be a chain module generated by 
$\Omega_{\nu}:=\{e^{\nu}_{1},\ldots,e^{\nu}_{k_{\nu}}\}$. Consider 
$\rho\in C_{\nu}$, 
$\rho=\sum_{i=1}^{k_{\nu}}a_{i}e^{\nu}_{i}$. Then we denote the 
coefficient of $e^{\nu}_{i}$ with $\rho(e^{\nu}_{i}):=a_{i}$.
\begin{rem}
  \label{rem:complexoforderzero}
  Consider a finite chain complex $(C,\Omega)$ of order $0$
  with basis $\Omega=\Omega_{0}=\{e^{0}_{1},\ldots,e^{0}_{k_{d}}\}$. 
  There are $(k_{0}-1)$ critical basis elements 
  and $H_{0}(C)\cong R^{k_{0}}$. 
\end{rem}

\begin{thm}
	Let $(C,\Omega)$ be a pure finite chain complex of order $d\geq 1$ and 
	$\Omega_{d}$ be a basis of $C_{d}$ with $k_{d}\geq 1$ elements. Let 
	there be $n<k_{d}$ critical elements $g_{1},\ldots,g_{n}$ in 
	$\Omega_{d}$ and all other basis elements be noncritical. 
	Let $\Omega_{d}$ be ordered in such a way that the 
	noncritical elements come first:
	\begin{equation*}
		\Omega_{d}=\{e_{1},\ldots,e_{m},g_{1},\ldots,g_{n}\},\qquad 
		m+n=k_{d}.
	\end{equation*}
	Then the following holds:
	\begin{enumerate}
		\item  $H_{d}(C) \cong R^{n}$.
	
		\item For $n\geq 1$, there exist unique cycles
		$\rho_{1},\ldots,\rho_{n}$ in $H_{d}(C)\cong\ker\bm_{d}$ so
		that $\rho_{i}(g_{j})=\delta_{ij}$.
	
		\item For $n\geq 1$, $\{\rho_{1},\ldots,\rho_{n}\}$ is a basis 
		of $H_{d}(C)$.
	\end{enumerate}
\end{thm}

\begin{proof}
	We start with the case $n=0$, \ie $\Omega_{d}$ has only noncritical 
	elements, so $\Omega_{d}=\{e_{1},\ldots,e_{m}\}$. 
	
	We assume: There exists an element $x\in C_{d}$ with 
	$x=\sum_{i=1}^{m}a_{i}e_{i}\neq 0$, so that $\bm_{d}(x)=0$. 
	
	As $x\neq 0$ there is some $a_{i}\neq 0$.
	We define 
	$i_{0}:=\max\set{i\leq m}{a_{i}\neq 0}$ and conclude
	$a_{i_{0}}\bm_{d}(e_{i_{0}})=\sum_{i<i_{0}}(-a_{i})\bm_{d}(e_{i})$, so $e_{i_{0}}$ 
	is not noncritical. {\contradiction } 
 	
	Hence, $\ker\bm_{d}=0$, \ie $H_{d}(C)=0$. 

	For $n\geq 1$, the first statement is a consequence of the second and third, 
	so we start proving the second statement using induction. Having $n\geq 1$ 
	critical elements, 
	$\Omega_{d}=\{e_{1},\ldots,e_{m},g_{1},\ldots,g_{n}\}$.
	
	We consider the chain complex $\widehat C := \bigcup_{i=1}^{m}C_{e_{i}}$, 
	whose chain modules are: $\widehat C_{d}=\erz{e_{1},\ldots,e_{m}}$, 
	$\widehat C_{\nu}=C_{\nu}$ for $d-1\geq\nu\geq 0$.
	The chain complex $\widehat C$ is a pure finite subcomplex of $C$ 
	of order $d$ without precritical elements, as the only precritical 
	elements in $\Omega_{d}$ are $g_{1},\ldots,g_{n}$. So 
	$H_{d}(\widehat C)=0$. 
	
	As the elements $g_{i}$ are all critical, there exists a 
	$\widehat\rho_{i}\in\widehat C_{d}$ for every $1\leq i\leq n$ so that 
	$\bm_{d}(\widehat\rho_{i})=\bm_{d}(g_{i})$. Let
	$\rho_{i}:=g_{i}-\widehat\rho_{i}$, then
	$\bm_{d}(\rho_{i})=\bm_{d}(g_{i})-\bm_{d}(\widehat\rho_{i})=0$;
	so $\rho_{i}\in\ker\bm_{d}\cong H_{d}(C)$. 
	
	Because $\widehat\rho_{i}\in\widehat C_{d}$ for every $1\leq i\leq n$, it 
	is $\widehat\rho_{i}=\sum_{\ell=1}^{m}a_{\ell}^{i}e_{\ell}$
	with $a_{\ell}^{i}\in R$. 
	We get:
	$\rho_{i}=g_{i} + \sum_{\ell=1}^{m}(-a_{\ell}^{i})e_{\ell}$. So there is 
	only $g_{i}$ with coefficient $1$ in $\rho_{i}$:
	\begin{equation*}
		\rho_{i}(g_{j})=\delta_{ij}=
		\begin{cases}
			1 & \text{if }i=j, \\
			0 & \text{if }i\neq j.
		\end{cases}
	\end{equation*}
	Therefore, we have proved the second statement up to uniqueness. 
	
	Consider $\sigma_{i}\in H_{d}(C)\cong\ker(\bm_{d})$ so that 
	$\sigma_{i}=g_{i}+\sum_{\ell =1}^{m}b_{\ell}^{i}e_{\ell}$.  
	Then we get $\sigma_{i}(g_{j})=\delta_{ij}$. We conclude
	$\sigma_{i}-\rho_{i}
	=\sum_{\ell=1}^{m}(b_{\ell}^{i}+a_{\ell}^{i})e_{\ell}$, so 
	$\sigma_{i}-\rho_{i}\in H_{d}(\widehat C)=0$. It follows 
	$\sigma_{i}=\rho_{i}$, so we have shown uniqueness. 
	
	\medskip
	Now we have to show that $\{\rho_{1},\ldots,\rho_{n}\}$ generates 
	$H_{d}(C)$ and is linearly independent. 
	
	Let $\sum_{i=1}^{n}a_{i}\rho_{i}=0$ with $a_{i}\in R$. Because 
	$\rho_{i}=g_{i}-\widehat\rho_{i}$, we get:
	\begin{equation*}
		0=\sum_{i=1}^{n}a_{i}g_{i} 
			- \underbrace{\sum_{i=1}^{n}a_{i}\widehat\rho_{i}.}_{
			\in\widehat C_{d}=\erz{e_{1},\ldots,e_{m}}}
	\end{equation*}
	As $\{e_{1},\ldots,e_{m},g_{1},\ldots,g_{n}\}$ is a basis of 
	$C_{d}$, we conclude $a_{i}=0$ for all $i$, so the elements 
	$\rho_{1},\ldots,\rho_{n}$ are independent. 
	
	We consider $\sigma\in H_{d}(C)\cong\ker(\bm_{d})\subset C_{d}$ 
	and define $\tau:=\sigma -\sum_{i=1}^{n}\sigma(g_{i})\rho_{i}$. 
	For $1\leq j\leq n$, the coefficient $\tau(g_{j})$ of $g_{j}$ in 
	$\tau$ is:
	\begin{equation*}
		\tau(g_{j}) = \sigma(g_{j}) - 
		\sum_{i=1}^{n}\sigma(g_{i})\underbrace{\rho_{i}(g_{j})}_{=\delta_{ij}}
		=\sigma(g_{j}) - \sigma(g_{j}) = 0.
	\end{equation*}
	Hence, $\tau$ is a combination of $e_{1},\ldots,e_{m}$, and we 
	conclude $\tau\in H_{d}(\widehat C)=0$, as 
	$\tau\in\ker\bm_{d}|_{\widehat C_{d}}$. So we get 
	$\sigma = \sum_{i=1}^{n}\sigma(g_{i})\rho_{i}$, \ie 
	$\{\rho_{1},\ldots,\rho_{n}\}$ generates $H_{d}(C)$.
	
	Hence, $\{\rho_{1},\ldots,\rho_{n}\}$ is a basis of $H_{d}(C)$, 
	and because of $\rho_{i}(g_{j})=\delta_{ij}$ we get 
	$H_{d}(C)\cong R^{n}$.
\end{proof}

In general, there are also precritical elements which are not critical. 
In this case we only know $H_{d}(C)\cong R^{n}$, but we cannot name a basis. 

\begin{thm}
	\label{thm:precritical}
	Let $(C,\Omega)$ be a pure finite chain complex of order $d\geq 1$ and 
	$\Omega_{d}$ be a basis of $C_{d}$ with $k_{d}\geq 1$ elements. Let 
	there be $n<k_{d}$ precritical elements $g_{1},\ldots,g_{n}$ in 
	$\Omega_{d}$. Let $\Omega_{d}$ be ordered in such a way that the 
	noncritical elements come first:
	\begin{equation*}
		\Omega_{d}=\{e_{1},\ldots,e_{m},g_{1},\ldots,g_{n}\},\qquad 
		m+n=k_{d}.
	\end{equation*}
	Then $H_{d}(C) \cong R^{n}$.
\end{thm}

\begin{proof}
	The case $n=0$ is already proven. 
	We consider the general case with $n\geq 1$ precritical 
	elements, so 
	$\Omega_{d}=\{e_{1},\ldots,e_{m},g_{1},\ldots,g_{n}\}$. 

	As above we consider the chain complex 
	$\widehat C := \bigcup_{i=1}^{m}C_{e_{i}}$ which is a pure  
	subcomplex of $C$ of order $d$ without precritical elements, hence 
	$H_{d}(\widehat C)=0$.

	As the elements $g_{i}$ are all precritical, there exist some 
	$\widehat\rho_{i}\in\widehat C_{d}$ and $0\neq a_{i}\in R$
	for every $1\leq i\leq n$ 
	so that $\bm_{d}(\widehat\rho_{i})=\bm_{d}(a_{i}g_{i})$. 
	We define 
	$\rho_{i}:=a_{i}g_{i}-\widehat\rho_{i}$ and conclude:
	$\bm_{d}(\rho_{i})=\bm_{d}(a_{i}g_{i})-\bm_{d}(\widehat\rho_{i})=0$.
	Hence $\rho_{i}\in\ker\bm_{d}\cong H_{d}(C)$. 
		
	We show that the elements $\rho_{1},\ldots,\rho_{n}$ are 
	linearly independent. 
	Let $\sum_{i=1}^{n}c_{i}\rho_{i}=0$ with $c_{i}\in R$. It is 
	$\rho_{i}=a_{i}g_{i}-\widehat\rho_{i}$, so we get:
	\begin{equation*}
		0=\sum_{i=1}^{n}c_{i}a_{i}g_{i} 
			- \underbrace{\sum_{i=1}^{n}c_{i}\widehat\rho_{i}.}_{
			\in\widehat C_{d}=\erz{e_{1},\ldots,e_{m}}}
	\end{equation*}
	As $\{e_{1},\ldots,e_{m},g_{1},\ldots,g_{n}\}$ is a basis of 
	$C_{d}$, we conclude $c_{i}a_{i}=0$ for all $i$, so $c_{i}= 0$. 
	Hence the elements 
	$\rho_{1},\ldots,\rho_{n}$ are independent. 
	
	\medskip
	Recall the generating number $\gen_{R}M$ and the degree of 
	freedom $\degf_{R}M$ which we have introduced in 
	Section~\ref{ssec:principalrings} for finitely generated $R$-modules $M$. 
	About the boundary mapping $\bm_{d}\colon C_{d}\rightarrow C_{d-1}$ 
	we know due to Theorem~\ref{thm:degree}:
	\begin{equation}
		\label{eq:degree}
		\degf_{R}(C_{d})=\degf_{R}(\ker\bm_{d}) + \degf_{R}(\im\bm_{d}).
	\end{equation}
	We have shown above that there are $n$ independent elements in 
	$\ker\bm_{d}$, hence $\degf_{R}(\ker\bm_{d})\geq n$. 
	
	$C_{d}$ is generated by $(m+n)$ elements. The equality 
	$\gen_{R}(C_{d})=\degf_{R}(C_{d})$ implies
	$\degf_{R}(C_{d})=m+n$. 
	
	Consider now the subcomplex $\widehat C$. Because all 
	$g_{i}$ are precritical, 
	$\im\bm_{d}\supset\im\bm_{d}|_{\widehat C_{d}}$. 
	By Theorem~\ref{thm:degree} we get:
	\begin{equation*}
		m=
		\degf_{R}(\widehat C_{d})
		= \degf_{R}\underbrace{(\ker\bm_{d}|_{\widehat C_{d}})}_{=0} 
		+ \degf_{R}(\im\bm_{d}|_{\widehat C_{d}})
		\leq \degf_{R}(\im\bm_{d}).
	\end{equation*}
	Applying Equation~\eqref{eq:degree} we conclude:
	\begin{equation*}
		\degf_{R}(\ker\bm_{d})=n
		\quad\text{and}\quad
		\degf_{R}(\im\bm_{d})=m. 
	\end{equation*}
	According to Theorem~\ref{thm:genequaldegree},
	$\gen_{R}(\ker\bm_{d})=n$, therefore $H_{d}(C)\cong R^{n}$. 
\end{proof}

\subsection{Reduced Homology}
Notice that there are chain complexes for which the augmentation homomorphism
$\epsilon$ must be $0$. 
For example, consider the chain complex of order $1$
over $\Z$ whose chain modules have the bases $\Omega_{1}=\{e^{1}_{1}\}$
and $\Omega_{0}=\{e^{0}_{1}\}$ with $\bm_{1}(e^{1}_{1})=e^{0}_{1}$. 

If there is a basis element $e^{0}_{j}\in\Omega_{0}$ which is not contained 
in the boundary of any basis element of $\Omega_{1}$ there is always a 
mapping $\epsilon\neq 0$ by defining $\epsilon(e^{0}_{j})=1$. In particular, 
$\epsilon$ can be defined this way for every finite chain complex of order $0$. 
For chain complexes of order $d\geq 1$ we treat a special case:

\begin{lemma}
  \label{lem:epsilon}
  Let $(C,\Omega)$ be a chain complex of order $d\geq 1$. 
  Let  $\Omega_{1}=\{e^{1}_{1},\ldots,e^{1}_{k_{1}}\}$ 
  and $\Omega_{0}=\{e^{0}_{1},\ldots,e^{0}_{k_{0}}\}$ 
  be bases of the chain modules $C_{1}$ and $C_{0}$.
  Let $\anzahl\bd(x)\geq 2$ for every $x\in C_{1}\setminus\ker\bm_{1}$. 
  Then a $R$-linear mapping $\epsilon\colon C_{0}\rightarrow R$ exists 
  such that $\epsilon(e^{0}_{\ell})\neq 0$ for all $1\leq \ell\leq k_{0}$.
\end{lemma}

\begin{proof}
  For any $e^{0}_{\ell}$ which is not contained in the boundary of 
  some $e^{1}_{j}$ we define $\epsilon(e^{0}_{\ell})=1$. Hence, we assume 
  without loss of generality that the $1$-skeleton of $C$ is pure. 
  Let
  \begin{equation*}
    \bm_{1}(e^{1}_{i})=\sum_{\ell=1}^{k_{0}} a_{i \ell}e^{0}_{\ell}
    \quad\text{for }1\leq i\leq k_{1}.
  \end{equation*}
  Because $\epsilon\circ\bm_{1}=0$ we have to solve the following system of 
  linear equations to define $\epsilon$:
  \begin{equation*}
    \label{eq:system}
    \begin{pmatrix}
      a_{11} & a_{12} & \ldots & a_{1 k_{0}} \\ 
      a_{21} & a_{22} & \ldots & a_{2 k_{0}} \\ 
      \vdots & \vdots & & \vdots \\
      a_{k_{1} 1} & a_{k_{1} 2} & \ldots & a_{k_{1} k_{0}}
    \end{pmatrix}
    \begin{pmatrix}
      \epsilon({e^{0}_{1}})\\ 
      \epsilon({e^{0}_{2}})\\ 
      \vdots \\
      \epsilon({e^{0}_{k_{0}}})\\ 
    \end{pmatrix}
    =0.
  \end{equation*}

  We assume that we get a line with only one entry $a_{ij}\neq 0$, \ie 
  there is an element $e^{0}_{j}$ with $\epsilon(e^{0}_{j})=0$. Getting such a 
  line means that there exists an element $x\in C_{1}$ so that 
  $\anzahl\bd(x)=1$ which is a contradiction!

  Therefore, the system of linear equations has a solution 
  $\bigl(\epsilon(e^{0}_{1}),\ldots,\epsilon(e^{0}_{k_{0}})\bigr)$ with all 
  $\epsilon(e^{0}_{\ell})\neq 0$.
\end{proof}

\section{Acyclic Chain Complexes and Cones}
\label{sec:acyclic}

\subsection{Terms and Definitions}

We define acyclic chain complexes in the same way as acyclic simplicial 
complexes.
\begin{defn}
	A chain complex $(C,\Omega)$ over a principal ring $R$ is 
	\emph{acyclic} if the following holds for the homology groups:
	\begin{center}
		$H_{0}(C)\cong R$, \qquad $H_{\nu}(C)=0$ \quad for $\nu\geq 1$. 
	\end{center}
\end{defn}

Some special simplicial complexes are \emph{cones}: A simplicial cone 
has a distinguished vertex $v_{0}$, and for any maximal simplex $S$ in 
the complex (\ie simplices which are not contained in any other) holds: 
$S$ has exact one facet which does not contain the vertex 
$v_{0}$. For example, simplices themselves are cones. 

To define the concept of a cone for chain complexes, we want to 
abandon the geometrical idea of an apex. 

\begin{defn}
	Let $(C,\Omega)$ be a finite chain complex of order $d$ over a 
	principal ring~$R$. 
	For $0\leq\nu\leq d$ let $\Omega_{\nu}$ be a basis of $C_{\nu}$, 
	$\Omega_{\nu}=\{e^{\nu}_{1},\ldots,e^{\nu}_{k_{\nu}}\}$.
	
	$(C,\Omega)$ is a \emph{cone} if the following conditions hold:
	\begin{enumerate}
		\item  \label{cone1} For every $\nu\in\{1,\ldots,d\}$ there is a nonempty 
		subset $S_{\nu}\subset\Omega_{\nu}$ so that:
		\begin{enumerate}
			\item  \label{cone1a} $\supp(\bm_{\nu} e^{\nu}_{j})\not\subset
			\bigcup\limits_{e^{\nu}_{i}\in S_{\nu}\setminus\{e^{\nu}_{j}\}}
			\supp(\bm_{\nu} e^{\nu}_{i})$
			for every $e^{\nu}_{j}\in S_{\nu}$;
		
			\item \label{cone1b} for every $e^{\nu}_{k}\in \Omega_{\nu}\setminus S_{\nu}$ 
			there is an element $\tau_{k}\in C_{\nu +1}$ so that 
			\begin{equation*}
				\bm_{\nu+1}\tau_{k}=ce^{\nu}_{k} + r
				\qquad\text{with $c$ unit in $R$ and $r\in\erz{S_{\nu}}$}.
			\end{equation*}
		\end{enumerate}
	
		\item  \label{cone2} $\anzahl{\supp(\bm_{1}\sigma)}\geq 2$ for every 
		$\sigma\in C_{1}\setminus\ker\bm_{1}$.
	
		\item  \label{cone3} There is a subset $\{e\}=S_{0}\subset\Omega_{0}$ 
		with $\anzahl{S_{0}}=1$ so that:
		\begin{itemize}
			\item[ ]  For every $e^{0}_{k}\in \Omega_{0}\setminus S_{0}$ 
			there is an element $\tau_{k}\in C_{1}$ so that 
			\begin{equation*}
				\bm_{1}\tau_{k}=ce^{\nu}_{k} + c_{0}e
				\qquad\text{with $c$ unit in $R$ and $c_{0}\neq 0$}.
			\end{equation*}
		\end{itemize}
	\end{enumerate}
\end{defn}

\begin{rem}
	\begin{enumerate}
		\item  $c_{0}\neq 0$ in condition~\ref{cone3} follows from condition~\ref{cone2}.
	
		\item  Recall the set $\Gamma$ from 
		Definition~\ref{defn:critical}. It is always
		$\Gamma\cap\Omega_{\nu}\subset S_{\nu}$. 
		In particular, $\Omega_{d}=S_{d}$. 
	
		\item  $(\ker\bm_{\nu})\cap\erz{S_{\nu}}=\{0\}$
		for all $1\leq\nu\leq d$
		because of condition~\ref{cone1a}.
	\end{enumerate}
\end{rem}

\begin{lemma}
	A cone $(C,\Omega)$ is acyclic.
\end{lemma}

\begin{proof}
	If $d=0$ (C finite complex of order $0$), then $\Omega_{0}=S_{0}$. 
	So we get $H_{\nu}(C)=0$ for $\nu\geq1$ and $H_{0}(C)\cong R$ 
	because of $\anzahl{S_{0}}=1$. 
	
	\medskip
	Consider now the case $d\geq 1$.	
	At first we show: $\ker\bm_{\nu}=\im\bm_{\nu -1}$ for $\nu\geq 1$.
	If $\nu>d$, there is nothing to do. 
	
	As $S_{d}=\Omega_{d}$ we conclude $\ker\bm_{d}=0$. Therefore, 
	$H_{d}(C)=0$. 
	
	For $1\leq\nu\leq d-1$, consider an element 
	$\sigma\in\ker\bm_{\nu}$: 
	\begin{equation*}
		\sigma = \sum_{i=1}^{k_{\nu}}a_{i}e^{\nu}_{i}
		= \sum_{e^{\nu}_{i}\in S_{\nu}}a_{i}e^{\nu}_{i} 
		+ \sum_{e^{\nu}_{i}\not\in S_{\nu}}a_{i}e^{\nu}_{i}.
	\end{equation*}
	There exists some $\tau_{i}\in C_{\nu +1}$
	for every $e^{\nu}_{i}\not\in S_{\nu}$ so that
	$\bm_{\nu+1}\tau_{i} = c_{i}e^{\nu}_{i} + r_{i}$ 
	with $c_{i}$ unit in $R$ and $r_{i}\in\erz{S_{\nu}}$,
	as postulated. 
	It is $\bm_{\nu+1}\tau_{i}\in\ker\bm_{\nu}$ because
	$\bm_{\nu}\circ\bm_{\nu+1}=0$. Therefore, we get:
	\begin{equation*}
		\underbrace{
			\sigma - \sum_{e^{\nu}_{i}\not\in S_{\nu}} (a_{i}c_{i}^{-1}) 
			\bm_{\nu+1}\tau_{i}
		}_{\in\ker\bm_{\nu}}
		=
		\underbrace{
			\sum_{e^{\nu}_{i}\in S_{\nu}} a_{i}e_{i}^{\nu}
			-
			\sum_{e^{\nu}_{i}\not\in S_{\nu}} a_{i}c_{i}^{-1} r_{i}
		}_{\in\erz{S_{\nu}}}
	\end{equation*}
	As $(\ker\bm_{\nu})\cap\erz{S_{\nu}}=\{0\}$, we conclude 
	$\sigma - \sum_{e^{\nu}_{i}\not\in S_{\nu}} (a_{i}c_{i}^{-1}) 
	\bm_{\nu+1}\tau_{i} = 0$. Hence 
	\begin{equation*}
		\sigma = \sum_{e^{\nu}_{i}\not\in S_{\nu}} (a_{i}c_{i}^{-1}) 
		\bm_{\nu+1}\tau_{i} \in\im\bm_{\nu+1}.
	\end{equation*}
	So $\ker\bm_{\nu}\subset\im\bm_{\nu+1}$. 
	Therefore, $\ker\bm_{\nu}=\im\bm_{\nu+1}$ and $H_{\nu}(C)=0$ for 
	$1\leq\nu\leq d-1$.
	
	\medskip
	Now we have to show $H_{0}(C)\cong R$. We know
	$\ker\bm_{0}= C_{0}=\erz{\Omega_{0}}$ with 
	$\Omega_{0}=\{e^{0}_{1},\ldots,e^{0}_{k_{0}}\}$. 
	
	Recall that $\anzahl{\supp(\bm_{1}\sigma)}\geq 2$ for every 
	$\sigma\in C_{1}\setminus\ker\bm_{1}$. Hence,  
	$\lambda e^{0}_{i}\not\in\im\bm_{1}$ for any $1\leq i\leq k_{0}$ 
	and every $0\neq\lambda\in R$. 
	
	As $\anzahl{S_{0}}=1$ we have $\anzahl{\Omega_{0}}\geq 1$. We look 
	at two cases separately:
	\begin{itemize}
		\item  $\anzahl{\Omega_{0}}=1$, so 
		$\Omega_{0}=\{e^{0}_{1}\}$. It is $\im\bm_{1}=0$ because of
		the cone condition~\ref{cone2}. So we get:
		\begin{equation*}
			H_{0}(C) = \quotient{\ker\bm_{0}}{\im\bm_{1}}
			= \quotient{\erz{e^{0}_{1}}}{0}
			\cong R 
		\end{equation*}
	
		\item  $\anzahl{\Omega_{0}}\geq 2$, so 
		$\Omega_{0}\setminus S_{0}\neq\emptyset$. 
		Without loss of generality, let
		$S_{0}=\{e^{0}_{1}\}$. 
		For every $e^{0}_{i}$ with $2\leq i\leq k_{0}$ there exists some 
		$\tau_{i}\in C_{1}$ so that
		\begin{equation*}
			\bm_{1}\tau_{i} = c_{i}e^{0}_{i} + d_{i}e^{0}_{1}
			\qquad\text{with $c_{i}$ unit in $R$ and $d_{i}\neq 0$.}
		\end{equation*}
 		Without loss of generality we can assume:
		\begin{equation*}
			\bm_{1}\tau_{i} = e^{0}_{i} + d_{i}e^{0}_{1}
			\qquad\text{with $d_{i}\neq 0$.}
		\end{equation*}
		These elements are all independent and in $(\im\bm_{1})$. As 
		$\lambda e^{0}_{1}\not\in\im\bm_{1}$ for every $0\neq\lambda\in 
		R$, we get:
		\begin{align*}
			H_{0}(C) &= \quotient{\ker\bm_{0}}{\im\bm_{1}} \\
			&= \quotient
			{\erz{e^{0}_{1}, e^{0}_{i}+d_{i}e^{0}_{1}\mid 2\leq i\leq k_{0}}}
			{\erz{e^{0}_{i}+d_{i}e^{0}_{1}\mid 2\leq i\leq k_{0}}} \\
			&\cong\erz{e^{0}_{1}}
			\cong R.
			\qedhere
		\end{align*}
	\end{itemize}
\end{proof}

For later purpose we have a look at a special chain complex 
$(C,\Omega)$ of order~$1$. 
Let its chain modules $C_{0}$ be finitely generated and $C_{1}$ generated by one 
element: 
\begin{equation*}
	C_{1}=\erz{e^{1}_{1}}, \qquad 
	C_{0}=\erz{e^{0}_{1},\ldots,e^{0}_{k}}
	\quad\text{with 
	$\bd(e^{1}_{1})=\{e^{0}_{1},\ldots,e^{0}_{k}\}$.} 
\end{equation*}
Therefore, $C=C_{e^{1}_{1}}$. What is known about the cardinality 
of the basis $\Omega_{0}$ if $(C,\Omega)$ is acyclic?

$C_{0}$ is generated by at least one element 
($k\geq 1$) if $C$ is acyclic. The image $\im\bm_{1}$ is a free 
submodule of $C_{0}$ 
generated by $\bm_{1}e^{1}_{1}=\sum_{i=1}^{k}a_{i}e^{0}_{i}$ with 
$a_{i}\neq 0$ for all $i$. 
We distinguish some cases:
\begin{description}
	\item[$k=1$]  $C_{0}=\erz{e^{0}_{1}}$, so 
	$\bm_{1}e^{1}_{1}=a_{1}e^{0}_{1}$ with $a_{1}\neq 0$. It 
	follows:
	\begin{equation*}
		H_{0}(C) = \quotient{C_{0}}{\im\bm_{1}}
		= \quotient{\erz{e^{0}_{1}}}{\erz{a_{1}e^{0}_{1}}}
		\not\cong R
		\quad\text{for $a_{1}\neq 0$. }
	\end{equation*}

	\item[$k\geq 3$]  $C_{0}=\erz{e^{0}_{1},\ldots,e^{0}_{k}}$, 
	$\im\bm_{1}=\erz{\sum_{i=1}^{k}a_{i}e^{0}_{i}}$. 
	
	We assume 
	$H_{0}(C)=\quotient{C_{0}}{\im\bm_{1}}\cong R$. Then any two 
	nonzero elements $[e^{0}_{\ell}]$, $[e^{0}_{j}]$, $\ell\neq j$, of 
	$\quotient{C_{0}}{\im\bm_{1}}$ are not independent, \ie there 
	exist elements $x,\/ y\in R\setminus\{0\}$ so that
	\begin{equation*}
		xe^{0}_{\ell} + ye^{0}_{j}
		= r\sum_{i=1}^{k}a_{i}e^{0}_{i}
		= \sum_{i=1}^{k}(ra_{i})e^{0}_{i}
		\in\im\bm_{1}.
	\end{equation*}
	As $\{e^{0}_{1},\ldots,e^{0}_{k}\}$ is a basis of $C_{0}$ and 
	$a_{i}\neq 0$ for all $i$ we conclude $r=0$. So $xe^{0}_{\ell} + 
	ye^{0}_{j}=0$ {\contradiction } -- a contradiction to the 
	independence of $e^{0}_{\ell}$ and $e^{0}_{j}$ in $C_{0}$. Hence, 
	$H_{0}(C)\not\cong R$.
\end{description}
The case $k=2$ remains. Indeed, it is possible to get 
$H_{0}(C)\cong R$ then. Assume 
$\bm_{1}e^{1}_{1}=a_{1}e^{0}_{1}+a_{2}e^{0}_{2}$ with $a_{2}$ unit 
in $R$. Then, $a_{2}^{-1}a_{1}e^{0}_{1}+e^{0}_{2}\in\im\bm_{1}$, 
and we get:
\begin{equation*}
	H_{0}(C)
	=\quotient{\erz{e^{0}_{1},e^{0}_{2}}}
	{\erz{a_{2}^{-1}a_{1}e^{0}_{1}+e^{0}_{2}}}
	\cong{\erz{e^{0}_{1}}}\cong R.
\end{equation*}

But it is not necessary that $a_{1}$ or $a_{2}$ in 
$\bm_{1}e^{1}_{1}=a_{1}e^{0}_{1}+a_{2}e^{0}_{2}$ is a unit. 
Take $R=\Z$ and $\bm_{1}e^{1}_{1}=2e^{0}_{1}+3e^{0}_{2}$. Then:
\begin{equation*}
	H_{0}(C) 
	= \quotient{\erz{e^{0}_{1}, e^{0}_{2}}}
	{\erz{2e^{0}_{1}+3e^{0}_{2}}}
	\cong \Z
\end{equation*}
as this factor module is generated by $e^{0}_{1}+e^{0}_{2}$. 

We summarize:
\begin{lemma}
	\label{lem:anzahlcnull}
	Let $(C,\Omega)$ be a pure chain complex of order $1$ over a 
	principal ring~$R$. 
	Let its chain modules $C_{0}$ be finitely generated and $C_{1}$ 
	generated by one element: 
	\begin{equation*}
		C_{1}=\erz{e^{1}_{1}}, \qquad 
		C_{0}=\erz{e^{0}_{1},\ldots,e^{0}_{k}}
		\quad\text{with 
		$\bd(e^{1}_{1})=\{e^{0}_{1},\ldots,e^{0}_{k}\}$.} 
	\end{equation*}
	If $(C,\Omega)$ is acyclic, then $C_{0}$ is generated by two 
	elements (so $k=2$).
\end{lemma}

The converse is not true. 
If $C_{0}$ is generated by two elements, then $C$ is
not necessarily acyclic. We take $R=\Z$ and 
$\bm_{1}e^{1}_{1}=2e^{0}_{1}+2e^{0}_{2}$ and get:
\begin{equation*}
	H_{0}(C) 
	= \quotient{\erz{e^{0}_{1}, e^{0}_{2}}}
	{\erz{2e^{0}_{1}+2e^{0}_{2}}}
	\not\cong \Z
\end{equation*}
as the factor module is not torsion free: 
$2[e^{0}_{1}+e^{0}_{2}]=[0|$. 

\medskip
We present a few examples of cones and acyclic chain complexes.
\begin{exs}
\label{exs:acyclexam}
  \begin{enumerate}
	\item  Consider a simplicial complex $D$ which is a cone in the common sense. 
	Then $D$ is also a cone according to our definition: 
	If the distinguished vertex of $D$ is $v_{0}$, 
	then we choose these basis elements for $S_{\nu}$ which 
	correspond to $\nu$-dimensional simplices containing the vertex 
	$v_{0}$. In particular, $S_{0}=\{v_{0}\}$. 

	There is no need to set $S_{\nu}$ in this way, as the following
	example shows.

	\smallskip
	\item  Consider a finite chain complex $(C,\Omega)$ of order $2$ 
	over $\Z$ 
	whose chain modules have the bases
	$\Omega_{2}=\{e^{2}_{1}\}$, 
	$\Omega_{1}=\{e^{1}_{1},e^{1}_{2},e^{1}_{3}\}$ and
	$\Omega_{0}=\{e^{0}_{1},e^{0}_{2},e^{0}_{3}\}$.
	Let there be the following boundary mappings:
	\begin{align*}
		\bm_{2}(e^{2}_{1}) & = e^{1}_{1}-e^{1}_{2}+e^{1}_{3}, 
		&\bm_{1}(e^{1}_{1}) & = e^{0}_{3}-e^{0}_{2},  \\
		&
		&\bm_{1}(e^{1}_{2}) & = e^{0}_{3}-e^{0}_{1}, \\
		&
		&\bm_{1}(e^{1}_{3}) & = e^{0}_{2}-e^{0}_{1}.
	\end{align*}
	We choose
	$S_{2}=\Omega_{2}=\{e^{2}_{1}\}$, 
	$S_{1}=\{e^{1}_{1},e^{1}_{2}\}$ and
	$S_{0}=\{e^{0}_{1}\}$.
	Then all cone conditions are fulfilled but 
	$e^{0}_{1}\not\in C_{e^{1}_{1}}$. 

	\smallskip
	\item  \label{exacyc3} Let $(C,\Omega)$ be a finite chain complex of order $2$ 
	over $\Z$ with bases 
	$\Omega_{2}=\{e^{2}_{1}\}$, 
	$\Omega_{1}=\{e^{1}_{1},e^{1}_{2}\}$ and
	$\Omega_{0}=\{e^{0}_{1},e^{0}_{2}\}$ so that:
	\begin{align*}
		\bm_{2}(e^{2}_{1}) & = e^{1}_{1}+e^{1}_{2}, 
		& \bm_{1}(e^{1}_{1}) & = e^{0}_{2}-e^{0}_{1},  \\
		&
		&\bm_{1}(e^{1}_{2}) & = e^{0}_{1}-e^{0}_{2}.
	\end{align*}
	The choice of 
	$S_{2}=\Omega_{2}=\{e^{2}_{1}\}$, 
	$S_{1}=\{e^{1}_{1}\}$ and
	$S_{0}=\{e^{0}_{1}\}$
	makes $(C,\Omega)$ a cone. 
	Notice that this chain complex does not come from a simplicial 
	complex! 

	\smallskip
	\item  \label{exacyc4} Again, we consider a finite chain complex $(C,\Omega)$ of 
	order $2$ over $\Z$. Let 
	$\Omega_{2}=\{e^{2}_{1},e^{2}_{2},e^{2}_{3}\}$, 
	$\Omega_{1}=\{e^{1}_{1},e^{1}_{2},e^{1}_{3},e^{1}_{4}\}$,
	$\Omega_{0}=\{e^{0}_{1},e^{0}_{2}\}$ and: 
	\begin{align*}
		\bm_{2}(e^{2}_{1}) & = e^{1}_{1}+e^{1}_{2},
		&\bm_{1}(e^{1}_{1}) & = e^{0}_{2}-e^{0}_{1},  \\
		\bm_{2}(e^{2}_{2}) & = e^{1}_{2}+e^{1}_{3}, 
		&\bm_{1}(e^{1}_{2}) & = e^{0}_{1}-e^{0}_{2},  \\
		\bm_{2}(e^{2}_{3}) & = e^{1}_{3}+e^{1}_{4},
		&\bm_{1}(e^{1}_{3}) & = e^{0}_{2}-e^{0}_{1}, \\
		&
		&\bm_{1}(e^{1}_{4}) & = e^{0}_{1}-e^{0}_{2}.
	\end{align*}
	We choose
	$S_{2}=\Omega_{2}=\{e^{2}_{1},e^{2}_{2},e^{2}_{3}\}$, 
	$S_{1}=\{e^{1}_{2}\}$ and
	$S_{0}=\{e^{0}_{1}\}$. 
	Because $\bm_{2}(e^{2}_{2}-e^{2}_{3})=e^{1}_{2}-e^{1}_{4}$ all 
	conditions for a cone are satisfied. 

	\smallskip
	\item  Our last example is also a finite chain complex $(C,\Omega)$ of 
	order $2$ over $\Z$. Let
	$\Omega_{2}=\{e^{2}_{1}\}$, 
	$\Omega_{1}=\{e^{1}_{1},e^{1}_{2},e^{1}_{3},e^{1}_{4}\}$,
	$\Omega_{0}=\{e^{0}_{1},e^{0}_{2},e^{0}_{3},e^{0}_{4}\}$  so that
	\begin{align*}
		\bm_{2}(e^{2}_{1}) & = e^{1}_{1}+e^{1}_{2}+e^{1}_{3}+e^{1}_{4}, 
		&\bm_{1}(e^{1}_{1}) & = e^{0}_{2}-e^{0}_{1},  \\
		&
		&\bm_{1}(e^{1}_{2}) & = e^{0}_{3}-e^{0}_{2},  \\
		&
		&\bm_{1}(e^{1}_{3}) & = e^{0}_{4}-e^{0}_{3},  \\
		&
		&\bm_{1}(e^{1}_{4}) & = e^{0}_{1}-e^{0}_{4}.
	\end{align*}
	The homology of $C$ is
	\begin{equation*}
		H_{2}(C)=0, \quad
		H_{1}(C)=0, \quad
		H_{0}(C)\cong\Z,
	\end{equation*}
	hence $C$ is acyclic. But $C$ is not a cone, which we will see as 
	follows:
	
	By definition holds:
		$\bd(e^{1}_{1})\cup\bd(e^{1}_{3})
		=\Omega_{0}
		=\bd(e^{1}_{2})\cup\bd(e^{1}_{4})$,
	and due to the cone condition~\ref{cone1a} we conclude
	$\anzahl S_{1}\leq 2$, hence 
	$\anzahl(\Omega_{1}\setminus S_{1})\geq 2$. 
	Because $\bd(e^{2}_{1})=\Omega_{1}$ it follows 
	$\anzahl{\bigl(\bd(e^{2}_{1})\cap(\Omega_{1}\setminus S_{1})\bigr)}\geq 2$, 
	and this is a contradiction to the cone condition~\ref{cone1b}.
\end{enumerate}
\end{exs}

\section{Shellable Chain Complexes}
\label{sec:shell}
In \citet[page~205]{BjornerLas-Vergnas.1999}, shellability is defined for 
\emph{regular cell complexes} which are more general than simplicial 
complexes. In a similar way, we define shellability of chain 
complexes. 

\subsection{Definition and First Examples}
\label{ssec:defshell}
Let $(C,\Omega)$ be a chain complex with basis $\Omega$. We define a 
mapping:
\begin{equation*}
	s\colon \Omega \rightarrow \Z;\qquad e^{\nu}_{i}\mapsto \nu = 
	\text{order of the complex $C_{e^{\nu}_{i}}$.}
\end{equation*}

\begin{defn}
	Let $(C,\Omega)$ be a finite chain complex of order $d$ over a 
	principal ring~$R$. Let
	$\Gamma := 
	\set{e\in\Omega}{e\not\in\bd(f) \text{ for all }f\in\Omega}
	\not=\emptyset$
	be the set of all maximal basis elements of $C$.
	An order of the basis elements in $\Gamma:=\{g_{1},\ldots,g_{k}\}$ 
	is a \emph{shelling} (or a \emph{shelling order}) if $d=0$ or if 
	the following holds for $d\geq 1$:
	\begin{enumerate}
		\item \label{shell1}  For $2\leq j\leq k$, the set 
		\begin{equation*}
			\Omega_{g_{j}}\cap 
			\Bigl(\bigcup_{i=1}^{j-1}\Omega_{g_{i}}\Bigr)
		\end{equation*}
		generates a pure chain complex of order $s(g_{j})-1$. 
	
		\item \label{shell2}  For $2\leq j\leq k$, the set 
		$(\Omega_{g_{j}})_{s(g_{j})-1}$ has a shelling in which the 
		basis elements of 
		$\Bigl(\Omega_{g_{j}}\cap
		(\bigcup_{i=1}^{j-1}\Omega_{g_{i}})\Bigr)_{s(g_{j})-1}$
		come first. 
	
		\item \label{shell3}  $(\Omega_{g_{1}})_{s(g_{1})-1}$ has a shelling.
	\end{enumerate}
	Then, the chain complex $(C,\Omega)$ is \emph{shellable}. 
\end{defn}

\begin{rem}
	\begin{enumerate}
		\item It must be $s(g_{1})=d$, otherwise it would be impossible to
		get a shelling because of condition~\ref{shell1}. So we can rewrite
		condition~\ref{shell3} as follows: $(\Omega_{g_{1}})_{d-1}$ has a shelling.
	
		\item  The definition of shellability for simplicial complexes 
		contains only condition~\ref{shell1} \citep[cf.][ch.~12]{Kozlov.2008a}. 
		As the boundary of a simplex is 
		always shellable, the conditions~\ref{shell2} and \ref{shell3} are trivially 
		satisfied. Hence, the definition above contains shellability of 
		simplicial complexes.
	
		\item  If $\Gamma=\{g_{1}\}$, only the third condition is 
		relevant. 
	
		\item  If $(C,\Omega)$ is a shellable chain complex of order 
		$d$, the chain modules of $C$ are $C_{\nu}\neq 0$
		for $0\leq\nu\leq d$. 
	
		\item  In a shellable chain complex $(C,\Omega)$, each 
		subcomplex $C_{e^{\nu}_{i}}$ is shellable, due to the 
		conditions~\ref{shell2} and \ref{shell3}. 
	
		\item  It follows from the conditions~\ref{shell2} and \ref{shell3} as well that the 
		$\bigl(s(g_{j})-1\bigr)$-skeleton of $C_{g_{j}}$ is shellable 
		for $1\leq j\leq k$. 
	
		\item  For a precritical element $g_{j}\in\Gamma$ it is
		\begin{equation*}
			\Omega_{g_{j}}\cap
			\Bigl(\bigcup_{i=1}^{j-1}\Omega_{g_{i}}\Bigr)
			=\Omega_{g_{j}}\setminus \{g_{j}\}. 
		\end{equation*}
		Therefore, it is possible to rearrange the elements in a 
		shelling of $\Gamma$ so that all precritical elements come at 
		last. 
	\end{enumerate}
\end{rem}

In opposite to shellable simplicial complexes the homology of 
shellable chain complexes is not clear. Consider the following 
examples:
\begin{exs}
\label{exs:shellexam}
\begin{enumerate}
	\item  \label{exshell1} Let $(C,\Omega)$ be a chain complex of order $1$ over $\Z$ 
		so that $C_{1}=\erz{e^{1}_{1}}$,  
		$C_{0}=\erz{e^{0}_{1},\ldots,e^{0}_{k}}$
		for some $k\geq 1$ and $\bm_{1}e^{1}_{1}=\sum_{i=1}^{k}e^{0}_{i}$. 
		This complex is shellable and its homology is:
		\begin{equation*}
			H_{1}(C)=0,\qquad H_{0}(C)\cong \Z^{k-1}.
		\end{equation*}
		If $k=2$, this chain complex is acyclic and even a cone. 

	\smallskip
	\item  \label{exshell2} Let $(C,\Omega)$ be a chain complex of order $1$ over $\Z$ 
		so that $C_{1}=\erz{e^{1}_{1}, e^{1}_{2}}$ and  
		$C_{0}=\erz{e^{0}_{1},\ldots,e^{0}_{k}}$
		for some $k\geq 1$. We assume: 
		$\bm_{1}e^{1}_{1}=\bm_{1}e^{1}_{2}=\sum_{i=1}^{k}e^{0}_{i}$. 
		This complex is shellable and its homology is:
		\begin{equation*}
			H_{1}(C)\cong \Z,\qquad H_{0}(C)\cong \Z^{k-1}.
		\end{equation*}
		
	\smallskip
	\item  \label{exshell3} Let $(C,\Omega)$ be a chain complex of order $1$ over $\Z$ 
		so that $C_{1}=\erz{e^{1}_{1}, e^{1}_{2}}$ and  
		$C_{0}=\erz{e^{0}_{1},\ldots,e^{0}_{k}}$
		for some $k\geq 2$. We assume: 
		$\bm_{1}e^{1}_{1}=\sum_{i=1}^{k}e^{0}_{i}$ and
		$\bm_{1}e^{1}_{2}=-e^{0}_{1}+\sum_{i=2}^{k}e^{0}_{i}$.
		This complex is shellable. About the homology we know:
		\begin{equation*}
			H_{1}(C)=0,\qquad H_{0}(C)\not\cong\Z^{i}\quad\text{for any 
			$i\in\N$}
		\end{equation*}
		as there are torsion elements in $H_{0}(C)$, for example $e^{0}_{1}$. 
\end{enumerate}
\end{exs}

So we need more conditions on shellable chain complexes to get some 
information about homology. We will treat this later 
in Section~\ref{sec:regular}.

\subsection{Monotonically Descending Shellings}
\label{ssec:mondescendshell}

Because of the first shelling condition the set
$\bigl(\Omega_{g_{j}}\cap 
(\bigcup_{i=1}^{j-1}\Omega_{g_{i}})\bigr)$ 
generates a pure chain complex of order $(s(g_{j})-1)$
for $2\geq j\geq k$.
We ask: Is it always possible to get a shelling of $\Gamma$ so that
$s(g_{i})\geq s(g_{i+1})$ for all $i$? Indeed, this is true, as we 
will show in this section. 

\begin{defn}
	\label{defn:mondescshell}
	Let $(C, \Omega)$ be a shellable chain complex over~$R$, 
	finite of order~$d$, and 
	$\Gamma\subset\Omega$ be the subset of all maximal basis elements. 
	A shelling of 
	$\Gamma:=\{g_{1},\ldots,g_{k}\}$ is \emph{monotonically descending} if 
	$s(g_{i})\geq s(g_{i+1})$ for all $1\leq i\leq k-1$. 
	
	A \emph{failure} in a shelling of $\Gamma=\{g_{1},\ldots,g_{k}\}$ 
	is a pair $(i,j)$ with $i<j$ and $s(g_{i})<s(g_{j})$. 
\end{defn}

Therefore, a monotonically descending shelling is a shelling without 
failures. 

\begin{rem}
	If $(C,\Omega)$ is a shellable chain complex and 
	\emph{pure}, then every shelling of $\Gamma$ is monotonically descending.
\end{rem}

At first, we will prove the following lemma:
\begin{lemma}
	\label{lem:mondescshell}
	Let $(C,\Omega)$ and $\Gamma$ be as in 
	Definition~\ref{defn:mondescshell}. Let there be $m\geq 1$ 
	failures in a shelling of $\Gamma$. Then it is possible to 
	permute the elements of $\Gamma$ so that there is a new shelling 
	of $\Gamma$ with $(m-1)$ failures. 
\end{lemma}

\begin{proof}
	Let $\Gamma:=\{g_{1},\ldots,g_{k}\}$, ordered in a shelling. 
	As $(C,\Omega)$ is shellable 
	we know $s(g_{1})\geq s(g_{i})$ for any $2\leq i\leq k$. 
	
	There is a minimal $2\leq i_{0}\leq k-1$ so that 
	$s(g_{i_{0}})<s(g_{i_{0}+1})$. 
	We want to show that we still have a shelling after permuting
	$g_{i_{0}}$ and $g_{i_{0}+1}$, \ie the ordered set 
	$\{g_{1},\ldots,g_{i_{0}-1},g_{i_{0}+1},g_{i_{0}},\ldots,g_{k}\}$  
	is also a shelling. 
	
	\medskip
	At first, we consider the chain complex 
	generated by:
	\begin{equation*}
		\Delta =
		\Bigl(\bigcup_{i=1}^{i_{0}}\Omega_{g_{i}}\Bigr)\cap 
		\Omega_{g_{i_{0}+1}}.
	\end{equation*}
	About this complex we know:
	\begin{itemize}
		\item  It is a pure chain complex of order 
		$(s(g_{i_{0}+1})-1)$. Hence, all maximal basis elements in $\Delta$ are in 
		$\bigl(\Omega_{g_{i_{0}+1}}\bigr)_{s(g_{i_{0}+1})-1}$. 
	
		\item  $\bigl(\Omega_{g_{i_{0}+1}}\bigr)_{s(g_{i_{0}+1})-1}$ 
		has a shelling in which the basis elements from $\Delta$
		come first. 
	\end{itemize}
	
	We divide the intersection into two parts:
	\begin{equation*}
		\Delta =
		\Biggl(
		\Bigl(\bigcup_{i=1}^{i_{0}-1}\Omega_{g_{i}}\Bigr)\cap 
		\Omega_{g_{i_{0}+1}}
		\Biggr)
		\cup
		\Bigl(
		\Omega_{g_{i_{0}}}\cap\Omega_{g_{i_{0}+1}}
		\Bigr).
	\end{equation*}	

	Because $s(g_{i_{0}})<s(g_{i_{0}+1})$, the set 
	$\Omega_{g_{i_{0}}}\cap \Omega_{g_{i_{0}+1}}$ 
	generates a chain complex of order 
	$t\leq s(g_{i_{0}})-1\leq s(g_{i_{0}+1})-2$.
	Therefore, any maximal basis element~$e$ in 
	$\Omega_{g_{i_{0}}}\cap \Omega_{g_{i_{0}+1}}$ 
	is contained in the boundary of some other basis element~$f\in\Delta$, 
	otherwise $\Delta$ would not generate a pure chain complex. 
	Because $e$ is maximal in 
	$\Omega_{g_{i_{0}}}\cap \Omega_{g_{i_{0}+1}}$ we conclude: 
	$f\not\in\Omega_{g_{i_{0}}}\cap \Omega_{g_{i_{0}+1}}$. 
	Hence, 
	$f\in\Bigl(\bigcup_{i=1}^{i_{0}-1}\Omega_{g_{i}}\Bigr)\cap 
	\Omega_{g_{i_{0}+1}}$. Therefore we conclude:
	\begin{equation*}
		\Omega_{g_{i_{0}}}\cap\Omega_{g_{i_{0}+1}}
		\subset
		\Bigl(\bigcup_{i=1}^{i_{0}-1}\Omega_{g_{i}}\Bigr)\cap 
		\Omega_{g_{i_{0}+1}},
		\quad\text{hence }
		\Delta =
		\Bigl(\bigcup_{i=1}^{i_{0}-1}\Omega_{g_{i}}\Bigr)\cap 
		\Omega_{g_{i_{0}+1}}.	
	\end{equation*}	
	
	Hence, the chain complex generated by 
	$\Bigl(\bigcup_{i=1}^{i_{0}-1}\Omega_{g_{i}}\Bigr)\cap 
	\Omega_{g_{i_{0}+1}}$ is pure of order $s(g_{i_{0}+1})-1$ and 
	satisfies all other shelling properties, too.
	
	\medskip
	We consider now the chain complex with basis
	\begin{equation*}
		\Lambda =
		\Biggl(
		\Bigl(\bigcup_{i=1}^{i_{0}-1}\Omega_{g_{i}}\Bigr)
		\cup\Omega_{g_{i_{0}+1}}
		\Biggr)
		\cap \Omega_{g_{i_{0}}}
		= 
		\Biggl(
		\Bigl(\bigcup_{i=1}^{i_{0}-1}\Omega_{g_{i}}\Bigr)
		\cap \Omega_{g_{i_{0}}}
		\Biggr)
		\cup
		\Bigl(
		\Omega_{g_{i_{0}+1}}\cap\Omega_{g_{i_{0}}}
		\Bigr).
	\end{equation*}
	Above we have shown 
	$\Omega_{g_{i_{0}}}\cap\Omega_{g_{i_{0}+1}}
	\subset
	\Bigl(\bigcup_{i=1}^{i_{0}-1}\Omega_{g_{i}}\Bigr)\cap 
	\Omega_{g_{i_{0}+1}}
	\subset
	\Bigl(\bigcup_{i=1}^{i_{0}-1}\Omega_{g_{i}}\Bigr)$. 
	As $\Omega_{g_{i_{0}}}\cap\Omega_{g_{i_{0}+1}}
	\subset\Omega_{g_{i_{0}}}$ we conclude:
	\begin{equation*}
		\Omega_{g_{i_{0}}}\cap\Omega_{g_{i_{0}+1}}
		\subset
		\Bigl(\bigcup_{i=1}^{i_{0}-1}\Omega_{g_{i}}\Bigr)\cap 
		\Omega_{g_{i_{0}}},
		\quad\text{hence }
		\Lambda =
		\Bigl(\bigcup_{i=1}^{i_{0}-1}\Omega_{g_{i}}\Bigr)\cap 
		\Omega_{g_{i_{0}}}.	
	\end{equation*}	
	Because the $g_{i}$ are ordered in a shelling, $\Lambda$ generates
	a pure chain complex of order $(s(g_{i_{0}})-1)$ which also satisfies all
	other shelling properties.
	
	Therefore, 
	$\{g_{1},\ldots,g_{i_{0}-1},g_{i_{0}+1},g_{i_{0}},\ldots,g_{k}\}$  
	is a shelling order with exactly one failure less. 
\end{proof}

By repeated application of this lemma we get:
\begin{thm}
	Let $(C,\Omega)$ be a shellable chain complex over~$R$, 
	finite of order~$d$, and 
	$\Gamma\subset\Omega$ be the subset of all maximal basis elements.
	Then a monotonically descending 
	shelling of $\Gamma=\{g_{1},\ldots,g_{k}\}$ exists.  
\end{thm}

\subsection{$i$-Skeletons of Shellable Chain Complexes}
\label{ssec:shellskeleton}
In Section~\ref{ssec:concepts} 
we introduced $i$-skeletons
$\sk_{i}(C)$ of a chain complex $(C,\Omega)$ as subcomplexes whose
chain modules $(\sk_{i}(C))_{\nu}$ are zero for $\nu>i$ and equal to 
$C_{\nu}$ otherwise. 

\begin{lemma}
	\label{lem:skeletonshell}
	Let $(C,\Omega)$ be a pure shellable chain complex, finite of order 
	$d\geq 1$. The $(d-1)$-skeleton $\sk_{d-1}(C)$ of $C$ is shellable, too. 
\end{lemma}

\begin{proof}
	Let $\Omega_{d}=\{e^{d}_{1},\ldots,e^{d}_{k_{d}}\}$. We have to 
	show: $\Omega_{d-1}=\bigcup_{i=1}^{k_{d}}\bd(e^{d}_{i})$ has a 
	shelling. We will do this by an inductive argument.
	
	By definition we know that
	$\bd(e^{d}_{1})=(\Omega_{e^{d}_{1}})_{d-1}$ has a shelling. 
	
	\medskip
	Consider now $\bigcup_{i=1}^{\ell}\bd(e^{d}_{i})\subset\Omega_{d-1}$ 
	for $2\leq\ell\leq k_{d}$. We assume that the set 
	$\bigcup_{i=1}^{\ell-1}\bd(e^{d}_{i})$ has a shelling, and we want to 
	show that
	\begin{equation*}
		\Bigl(\bigcup_{i=1}^{\ell-1}\bd(e^{d}_{i})\Bigr)
		\cup (\Omega_{e^{d}_{\ell}})_{d-1}
		=
		\Bigl(\bigcup_{i=1}^{\ell-1}\bd(e^{d}_{i})\Bigr)
		\cup \bd(e^{d}_{\ell})		
	\end{equation*}
	has a shelling, too. 
	
	\medskip
	Because the complex $C$ is shellable we know that
	$(\Omega_{e^{d}_{\ell}})_{d-1}$ has a shelling in which the
	elements of
	\begin{equation*}
		\Bigl(\bigcup_{i=1}^{\ell-1}\bd(e^{d}_{i})\Bigr)\cap\Omega_{e^{d}_{\ell}}
		=\Biggl(
		\Bigl(\bigcup_{i=1}^{\ell-1}\Omega_{e^{d}_{i}}\Bigr)\cap\Omega_{e^{d}_{\ell}}
		\Biggr)_{d-1}
	\end{equation*}
	come first. If $e^{d}_{\ell}$ is precritical, then 
	$\bigl(\bigcup_{i=1}^{\ell-1}\bd(e^{d}_{i})\bigr)\cap\Omega_{e^{d}_{\ell}}
	=\bigl(\bigcup_{i=1}^{\ell-1}\bd(e^{d}_{i})\bigr)$, and we are 
	done. 
	
	So we assume $e^{d}_{\ell}$ noncritical. Then let
	$(\Omega_{e^{d}_{\ell}})_{d-1}=\{f_{1},\ldots,f_{s},h_{1},\ldots,h_{t}\}$
	be a shelling so that 
	$\bigl(\bigcup_{i=1}^{\ell-1}\bd(e^{d}_{i})\bigr)\cap\Omega_{e^{d}_{\ell}}
	=\{f_{1},\ldots,f_{s}\}$ and $s\geq 1$. 
	As $e^{d}_{\ell}$ is noncritical we also have $t\geq 1$. 
	
	Let 
	$\bigcup_{i=1}^{\ell-1}\bd(e^{d}_{i})
	=\{e_{1},\ldots,e_{r},f_{1},\ldots,f_{s}\}$. This order is not 
	necessarily a shelling but this does not matter. 
	
	We want to show that 
	$\{e_{1},\ldots,e_{r},f_{1},\ldots,f_{s},h_{1},\ldots,h_{t}\}$ has 
	a shelling. We start by checking
	\begin{equation*}
		\Biggl(
		\Bigl( \bigcup_{i=1}^{r}\Omega_{e_{i}} \Bigr)
		\cup
		\Bigl( \bigcup_{i=1}^{s}\Omega_{f_{i}} \Bigr)	
		\Biggr)
		\cap
		\Omega_{h_{1}}
		=
		\Biggl(
		\Bigl( \bigcup_{i=1}^{r}\Omega_{e_{i}} \Bigr)
		\cap\Omega_{h_{1}}
		\Biggr)
		\cup
		\Biggl(
		\Bigl( \bigcup_{i=1}^{s}\Omega_{f_{i}} \Bigr)
		\cap\Omega_{h_{1}}
		\Biggr).
	\end{equation*}
	As 
	$(\Omega_{e^{d}_{\ell}})_{d-1}=\{f_{1},\ldots,f_{s},h_{1},\ldots,h_{t}\}$
	is a shelling the set
		$\bigl(
		\bigcup_{i=1}^{s}\Omega_{f_{i}}
		\bigr)\cap\Omega_{h_{1}}$
	generates a pure chain complex of order $(d-2)$ and satisfies all 
	shelling conditions. 
	
	Because
	$\bigl(\bigcup_{i=1}^{\ell-1}\Omega_{e^{d}_{i}}\bigr)\cap\Omega_{e^{d}_{\ell}}
	=\bigcup_{j=1}^{s}\Omega_{f_{j}}$
	we get
		$\bigl( \bigcup_{i=1}^{r}\Omega_{e_{i}} \bigr)
		\cap\Omega_{h_{1}}
		\subset\bigcup_{j=1}^{s}\Omega_{f_{j}}$
	and therefore 
		$\bigl( \bigcup_{i=1}^{r}\Omega_{e_{i}} \bigr)
		\cap\Omega_{h_{1}}
		\subset\bigcup_{j=1}^{s}\Omega_{f_{j}}\cap\Omega_{h_{1}}$.
	We conclude:
	\begin{equation*}
		\Biggl(
		\Bigl( \bigcup_{i=1}^{r}\Omega_{e_{i}} \Bigr)
		\cup
		\Bigl( \bigcup_{i=1}^{s}\Omega_{f_{i}} \Bigr)	
		\Biggr) 
		\cap
		\Omega_{h_{1}}
		=
		\Bigl( \bigcup_{i=1}^{s}\Omega_{f_{i}} \Bigr)
		\cap\Omega_{h_{1}},
	\end{equation*}
	so the set 
	$\{e_{1},\ldots,e_{r},f_{1},\ldots,f_{s},h_{1}\}$ has a shelling. 
	
	\medskip
	For any $1\leq j\leq t-1$ we assume the set 
	$\{e_{1},\ldots,e_{r},f_{1},\ldots,f_{s},h_{1},\ldots,h_{j}\}$ has 
	a shelling, \ie the set 
	\begin{equation*}
		\Bigl( \bigcup_{i=1}^{r}\Omega_{e_{i}} \Bigr)
		\cup
		\Bigl( \bigcup_{i=1}^{s}\Omega_{f_{i}} \Bigr)
		\cup
		\Bigl( \bigcup_{i=1}^{j}\Omega_{h_{i}} \Bigr)
	\end{equation*}
	generates a shellable chain complex of order $(d-1)$. Then a similar 
	argument as above shows that
	\begin{align*}
		&\Biggl(
		\Bigl( \bigcup_{i=1}^{r}\Omega_{e_{i}} \Bigr)
		\cup
		\Bigl( \bigcup_{i=1}^{s}\Omega_{f_{i}} \Bigr)	
		\cup
		\Bigl( \bigcup_{i=1}^{j}\Omega_{h_{i}} \Bigr)
		\Biggr) 
		\cap
		\Omega_{h_{j+1}} \\
		&=
		\underbrace{
		\Biggl(
		\Bigl( \bigcup_{i=1}^{r}\Omega_{e_{i}} \Bigr)
		\cap\Omega_{h_{j+1}}
		\Biggr)
		}_{
		\subset 
		\bigl( \bigcup_{i=1}^{s}\Omega_{f_{i}} \bigr)
		\cap\Omega_{h_{j+1}}
		}
		\cup
 		\underbrace{
		\Biggl(
		\Biggl(
		\Bigl( \bigcup_{i=1}^{s}\Omega_{f_{i}} \Bigr)
		\cup
		\Bigl( \bigcup_{i=1}^{j}\Omega_{h_{i}} \Bigr)
		\Biggr)
		\cap\Omega_{h_{j+1}}
		\Biggr)
 		}_{\text{generates a pure complex of order $(d-2)$}} 
		\\
		& =
		\Biggl(
		\Bigl( \bigcup_{i=1}^{s}\Omega_{f_{i}} \Bigr)
		\cup
		\Bigl( \bigcup_{i=1}^{j}\Omega_{h_{i}} \Bigr)
		\Biggr)
		\cap\Omega_{h_{j+1}}. 
	\end{align*}
	This set is a basis of a shellable pure chain complex of order 
	$(d-2)$. Therefore, the set
	$\{e_{1},\ldots,e_{r},f_{1},\ldots,f_{s},h_{1},\ldots,h_{j+1}\}$ 
	has a shelling. 
	
	By induction we conclude that the set 
	\begin{equation*}
		\{e_{1},\ldots,e_{r},f_{1},\ldots,f_{s},h_{1},\ldots,h_{t}\}
		=\bigcup_{i=1}^{\ell}\bd(e^{d}_{i})
	\end{equation*}
	has a shelling. 
	Therefore, the first induction delivers:
	$\Omega_{d-1}=\bigcup_{i=1}^{k_{d}}\bd(e^{d}_{i})$
	has a shelling. 
\end{proof}

\begin{thm}
	\label{thm:skeletonshell}
	Let $(C,\Omega)$ be a shellable chain complex, finite of 
	order~$d\geq 1$. 
	The $(d-1)$-skeleton $\sk_{d-1}(C)$ of $C$ is shellable, too. 
\end{thm}

\begin{proof}
	We have already proved this statement if $(C,\Omega)$ is pure. 
	So we have only to consider the non-pure case. 
	
	Let $\Gamma\subset\Omega$ be the set of all maximal basis elements
	and let the elements of $\Gamma$ be ordered 
	in a monotonically descending shelling.
	
	Let $\Omega_{d}=\{e^{d}_{1},\ldots,e^{d}_{k_{d}}\}\subset\Gamma$, 
	and 
	$\Gamma\cap\Omega_{d-1}=\{g^{d-1}_{1},\ldots,g^{d-1}_{m}\}$. In 
	the chosen shelling order of $\Gamma$ the elements of 
	$\Omega_{d}$ come first, followed by all elements of 
	$\Gamma\cap\Omega_{d-1}$. 
	The basis of the chain module $\bigl(\sk_{d-1}(C)\bigr)_{d-1}$ is
	\begin{equation*}
		\Omega_{d-1}=\Bigl(\bigcup_{i=1}^{k_{d}}\bd(e^{d}_{i})\Bigr)
		\cup\{g^{d-1}_{1},\ldots,g^{d-1}_{m}\}.
	\end{equation*}
	We have already proved that 
	$\bigl(\bigcup_{i=1}^{k_{d}}\bd(e^{d}_{i})\bigr)$ has a shelling 
	because the subcomplex generated by 
	$\bigcup_{i=1}^{k_{d}}\Omega_{e^{d}_{i}}$ is shellable and pure. 
	
	Let 
	$\widehat\Omega_{e^{d}_{i}}:=\Omega_{e^{d}_{i}}\setminus\{e^{d}_{i}\}$
	be the basis of the $(d-1)$-skeleton of $C_{e^{d}_{i}}$. Then we get:
	\begin{equation*}
		\Bigl(\bigcup_{i=1}^{k_{d}}\Omega_{e^{d}_{i}}\Bigr)
		\cap\Omega_{g^{d-1}_{1}}
		=
		\Bigl(\bigcup_{i=1}^{k_{d}}\widehat\Omega_{e^{d}_{i}}\Bigr)
		\cap\Omega_{g^{d-1}_{1}}. 
	\end{equation*}
	As $C$ is shellable, the set on the left side satisfies all 
	shelling conditions. Therefore, 
	$\Bigl(\bigcup_{i=1}^{k_{d}}\bd(e^{d}_{i})\Bigr)
	\cup\{g^{d-1}_{1}\}\subset\Omega_{d-1}$ 
	has a shelling, and by induction we 
	conclude that $\Omega_{d-1}$ has a shelling, too. 
\end{proof}

\begin{cor}
	Let $(C,\Omega)$ be a shellable chain complex of order $d$.
	For $0\leq i\leq d$, every $i$-skeleton $\sk_{i}(C)$ of $C$ is
	shellable.
\end{cor}

Hence, if any $i$-skeleton $\sk_{i}(C)$ of a chain complex $C$ is not shellable, 
the complex itself is not shellable. 

\medskip
In the proof of Lemma~\ref{lem:skeletonshell} we used a special ordering
of the chain module bases $\Omega_{\nu}$.  
We emphasise it for later purpose: 
\begin{rem}
  \label{rem:specialordering}
  Let $(C,\Omega)$ be a shellable chain complex of order $d$ and let 
  its set $\Gamma$ of maximal basis elements be ordered in a 
  monotonically descending shelling (\siehe{Definition~\ref{defn:mondescshell}}). 
  Then, we can order the elements in the bases
  $\Omega_{\nu}$ of the chain modules $C_{\nu}$, $0\leq \nu\leq d-1$, 
  as follows: 

  The first segment in $\Omega_{\nu}$ are the elements of $\bd(e^{\nu+1}_{1})$,
  ordered in a shelling. Then we add the elements of 
  $\bd(e^{\nu+1}_{2})\setminus\bd(e^{\nu+1}_{1})$
  in the same order as in the shelling of $\bd(e^{\nu+1}_{2})$. 
  We proceed iteratively, adding the elements of
  $\bd(e^{\nu+1}_{i})$ which are not contained in any 
  $\bd(e^{\nu+1}_{\ell})$ for $\ell < i$. 
  Eventually we are left with the elements of $\Omega_{\nu}\cap\Gamma$, 
  which we will add in the same order as they occur in $\Gamma$. 
  As proven above, this ordering delivers a shelling of $\Omega_{\nu}$. 
\end{rem}

\section{Regular Chain Complexes}
\label{sec:regular}

\subsection{Definition and Examples}
\label{ssec:regulardef}
\begin{defn}
	Let $(C,\Omega)$ be a shellable chain complex of order~$d$ over a 
	principal ring $R$ and $\Gamma\subset\Omega$ be the set of all 
	maximal basis elements. 
	
	Let $\Omega_{d}:=\{e^{d}_{1},\ldots,e^{d}_{k_{d}}\}$ 
	and
	$\Omega_{\nu}:=\{e^{\nu}_{1},\ldots,e^{\nu}_{m_{\nu}},
	e^{\nu}_{m_{\nu}+1},\ldots,e^{\nu}_{k_{\nu}}\}$
	for $\nu\leq d-1$
	so that the following holds\footnote{Notice two special cases: 
	\begin{itemize}
		\item  $m_{\nu}=k_{\nu}$. Then 
		$\Gamma\cap\Omega_{\nu}=\emptyset$. 
	
		\item  $m_{\nu}=0$. Then $\Gamma\cap\Omega_{\nu}=\Omega_{\nu}$. 
		This is always valid for $\Omega_{d}$. 
	\end{itemize}
	}:
	\begin{itemize}
		\item  
		$\Gamma\cap\Omega_{\nu}=\{e^{\nu}_{m_{\nu}+1},\ldots,e^{\nu}_{k_{\nu}}\}$;
	
		\item  $\Gamma=\{e^{d}_{1},\ldots,e^{d}_{k_{d}}\}
		\cup
		\{e^{d-1}_{m_{d-1}+1},\ldots,e^{d-1}_{k_{d-1}}\}
		\cup
		\ldots
		\cup
		\{e^{0}_{m_{0}+1},\ldots,e^{0}_{k_{0}}\}$
		is a monotonically descending shelling; 
	
		\item  for $0\leq \nu\leq d-1$, let each basis $\Omega_{\nu}$
		be ordered as in Remark~\ref{rem:specialordering}. 
	\end{itemize}
	
	Then, $\Gamma$ has a \emph{regular order} if the following
	two conditions are fulfilled:
	\begin{enumerate}
		\item  For any $e^{\nu}_{\ell}\in\Gamma$ (\ie 
		$m_{\nu}+1\leq\ell\leq k_{\nu}$):
		
		If $\bd(e^{\nu}_{\ell})\subset 
		\bigcup_{i=1}^{\ell-1}\bd(e^{\nu}_{i})$,
		then $e^{\nu}_{\ell}$ is precritical,
		\ie there exist elements $a_{i}\in R$, $1\leq i\leq \ell$, $a_{\ell}\neq 0$, 
		so that
		\begin{equation*}
			a_{\ell}\bm_{\nu}(e^{\nu}_{\ell})
			=\sum_{i=1}^{\ell-1}a_{i}\bm_{\nu}(e^{\nu}_{i}).
		\end{equation*}
	
		\item For any $e^{\nu}_{\ell}\in\Omega_{\nu}$, $\nu\geq 1$, let
		$(\Omega_{e^{\nu}_{\ell}})_{\nu-1}=\bd(e^{\nu}_{\ell}):=
		\{f^{\nu-1}_{1},\ldots,f^{\nu-1}_{n_{\ell}}\}$
		be a shelling
		in which the elements of 
		$(\bigcup_{i=1}^{\ell-1}\Omega_{e^{\nu}_{i}})
		\cap\Omega_{e^{\nu}_{\ell}}$
		come first
		so that for any $1\leq j\leq n_{\ell}$ holds:
		
		If $\bd(f^{\nu-1}_{j})\subset 
		\bigcup_{i=1}^{j-1}\bd(f^{\nu-1}_{i})$,
		then $c_{j}\bm_{\nu-1}(f^{\nu-1}_{j})$ is a linear combination of 
		$\bm_{\nu-1}(f^{\nu-1}_{i})$ for some $0\neq c_{j}\in R$, 
		\ie there exist 
		$c_{i}\in R$, $1\leq i\leq j$, $c_{j}\neq 0$, so that
		\begin{equation*}
			c_{j}\bm_{\nu-1}(f^{\nu-1}_{j})
			=\sum_{i=1}^{j-1}c_{i}\bm_{\nu-1}(f^{\nu-1}_{i}).
		\end{equation*}
	\end{enumerate}
	The chain complex $(C,\Omega)$ is \emph{regular} if the set $\Gamma$ 
	has a regular order. 
\end{defn}

\begin{defn}
	A regular chain complex $(C,\Omega)$ whose subcomplexes 
	$C_{e^{\nu}_{i}}$ are all acyclic is called \emph{totally regular}.
\end{defn}

\begin{rem}
	A chain complex which comes from a simplicial complex is always 
	totally regular. This is caused by the geometry of a simplex and 
	the special boundary mapping. 
\end{rem}

Any finite chain complex of order $0$ satisfies trivially all 
regularity conditions, so it is totally regular. But there also exist 
more serious examples:

\begin{exs}
	\begin{enumerate}
		\item  Let $(C,\Omega)$ be a finite chain complex of order $1$ 
		over $\Z$  
		whose chain modules have the bases 
		$\Omega_{1}=\{e^{1}_{1},e^{1}_{2}\}$ and 
		$\Omega_{0}=\{e^{0}_{1},e^{0}_{2}\}$. 
		Let $\bm_{1}(e^{1}_{1})=2e^{0}_{1}+e^{0}_{2}$
		and $\bm_{1}(e^{1}_{2})=e^{0}_{1}+2e^{0}_{2}$. 
		The reader may convince himself that this chain complex is 
		shellable. 
		
		It is $\bd(e^{1}_{1})=\Omega_{0}=\bd(e^{1}_{2})$, but 
		$\bm_{1}(e^{1}_{1})$ and $\bm_{1}(e^{1}_{2})$ are linearly 
		independent. So $(C,\Omega)$ is not regular. 
		
		Furthermore, it is 
		$\bm_{1}(2e^{1}_{1}-e^{1}_{2})=3e^{0}_{1}$. As 
		$e^{0}_{1}\not\in\im\bm_{1}$, the factor module 
		$H_{0}(C)=\quotient{C_{0}}{\im\bm_{1}}$ is not torsion free. 
		Hence, $H_{0}(C)\not\cong\Z$, \ie the chain complex 
		$(C,\Omega)$ is not acyclic. 

		\smallskip
		\item  We consider a finite chain complex $(C,\Omega)$ of order 
		$2$ over $\Z$ whose chain modules have the following bases:
		$\Omega_{2}=\{e^{2}_{1},e^{2}_{2}\}$, 
		$\Omega_{1}=\{e^{1}_{1},e^{1}_{2},e^{1}_{3}\}$ and
		$\Omega_{0}=\{e^{0}_{1},e^{0}_{2}\}$. 
		Let the boundary mappings $\bm_{2}$ and $\bm_{1}$ defined by:
		\begin{align*}
			\bm_{2}(e^{2}_{1}) & = 2e^{1}_{1}+e^{1}_{2}+e^{1}_{3}, 
			& \bm_{1}(e^{1}_{1}) & = e^{0}_{1}-e^{0}_{2},  \\
			\bm_{2}(e^{2}_{2}) & = e^{1}_{1}+e^{1}_{2}, 
			& \bm_{1}(e^{1}_{2}) & = e^{0}_{2}-e^{0}_{1},  \\
			&
			&\bm_{1}(e^{1}_{3}) & = e^{0}_{2}-e^{0}_{1}.
		\end{align*}
		This chain complex is shellable. But its natural order is 
		not regular because $\bd(e^{2}_{2})\subset\bd(e^{2}_{1})$.
		Changing the order of 
		$e^{2}_{1}$ and $e^{2}_{2}$ delivers a regular order, indeed! 
		Therefore, this chain complex is regular, but not totally regular
		as $C_{e^{2}_{1}}$ is not acyclic. 

		Now we have a look at the homology of $(C,\Omega)$. 
		\begin{itemize}
			\item  $\ker\bm_{2}=0$, so $H_{2}(C)=0$.
		
			\item  It is 
			$\ker\bm_{1}=\erz{(e^{1}_{1}+e^{1}_{2}),(e^{1}_{1}+e^{1}_{3})}
			=\im\bm_{2}$, hence $H_{1}(C)=0$.
		
			\item  As $\ker\bm_{0}=\erz{e^{0}_{1},e^{0}_{2}}$ and 
			$\im\bm_{1}=\erz{e^{0}_{1}-e^{0}_{2}}$ we get 
			$H_{0}(C)\cong\Z$. 
		\end{itemize}
		Therefore, $(C,\Omega)$ is acyclic. It is even a cone if we choose
		$S_{2}=\{e^{2}_{1},e^{2}_{2}\}$,
		$S_{1}=\{e^{1}_{1}\}$ and
		$S_{0}=\{e^{0}_{1}\}$. 
		
		\smallskip
		\item Let $(C,\Omega)$ be a finite chain complex of order $1$ 
		over $\Z$
		whose chain modules have the bases 
		$\Omega_{1}=\{e^{1}_{1},e^{1}_{2}\}$ and 
		$\Omega_{0}=\{e^{0}_{1},e^{0}_{2},e^{0}_{3}\}$. 
		Let $\bm_{1}(e^{1}_{1})=2e^{0}_{1}+e^{0}_{2}$
		and $\bm_{1}(e^{1}_{2})=e^{0}_{1}+e^{0}_{2}$. 
		
		This complex is shellable, and every subcomplex $C_{e^{\nu}_{j}}$ 
		is acyclic. Furthermore, 
		we have $\bd(e^{1}_{1})=\{e^{0}_{1},e^{0}_{2}\}=\bd(e^{1}_{2})$, but 
		$\bm_{1}(e^{1}_{1})$ and $\bm_{1}(e^{1}_{2})$ are linearly 
		independent. So $(C,\Omega)$ is not regular. 
		
		We compute the homology groups:
		\begin{itemize}
			\item  $\ker\bm_{1}=0$, so $H_{1}(C)=0$.
			
			\item  $\ker\bm_{0}=\erz{e^{0}_{1},e^{0}_{2},e^{0}_{3}}$ and 
			$\im\bm_{1}=\erz{e^{0}_{1}+e^{0}_{2},2e^{0}_{1}+e^{0}_{2}}$, therefore we get 
			$H_{0}(C)\cong\Z$. 
		\end{itemize}
		Hence, this chain complex is acyclic. But it is not a cone as
		$\bm_{1}(e^{1}_{1}-e^{1}_{2})=e^{0}_{1}$.
		
		\smallskip
		\item  Let $(C,\Omega)$ be a finite chain complex of order $2$ 
		over $\Z$ and $\Omega_{2}=\{e^{2}_{1}\}$, 
		$\Omega_{1}=\{e^{1}_{1},e^{1}_{2},e^{1}_{3}\}$, 
		$\Omega_{0}=\{e^{0}_{1},e^{0}_{2}\}$
		be the bases of its chain modules. Let
		\begin{align*}
			\bm_{2}(e^{2}_{1}) & = e^{1}_{1}+e^{1}_{2}+e^{1}_{3}, 
			& \bm_{1}(e^{1}_{1}) & = e^{0}_{1}-e^{0}_{2},  \\
			 &  
			& \bm_{1}(e^{1}_{2}) & = e^{0}_{1}-e^{0}_{2},  \\
			 &
			&\bm_{1}(e^{1}_{3}) & = 2(e^{0}_{2}-e^{0}_{1}).
		\end{align*}
		This chain complex is shellable and regular. 
		
		We compute its homology groups:
		\begin{itemize}
			\item $\ker\bm_{2}=0$, so $H_{2}(C)=0$.
		
			\item  As 
			$\ker\bm_{1}
			=\erz{(e^{1}_{1}-e^{1}_{2}),(e^{1}_{1}+e^{1}_{2}+e^{1}_{3})}$
			and
			$\im\bm_{2}=\erz{e^{1}_{1}+e^{1}_{2}+e^{1}_{3}}$ 
			we get $H_{1}(C)\cong\Z$.
		
			\item  $\ker\bm_{0}=\erz{e^{0}_{1},e^{0}_{2}}$ and 
			$\im\bm_{1}=\erz{e^{0}_{1}-e^{0}_{2}}$, hence 
			$H_{0}(C)\cong\Z$. 
		\end{itemize}
		Hence, $(C,\Omega)$ is not acyclic. In particular, $(C,\Omega)$ 
		is not totally regular because then $C_{e^{2}_{1}}=C$ must be 
		acyclic. 

		\smallskip
		\item  Let $(C,\Omega)$ be a finite chain complex of order $1$ 
		over $\Z$ whose chain modules have the bases
		$\Omega_{1}=\{e^{1}_{1},e^{1}_{2}\}$ and 
		$\Omega_{0}=\{e^{0}_{1},e^{0}_{2},e^{0}_{3}\}$. Let
		\begin{align*}
			\bm_{1}(e^{1}_{1}) & = e^{0}_{1},  \\
			\bm_{1}(e^{1}_{2}) & = e^{0}_{1}+e^{0}_{2}+e^{0}_{3}.
		\end{align*}
		We observe that this chain complex is acyclic, shellable and 
		regular, but not a cone.
		As both subcomplexes $C_{e^{1}_{1}}$ and $C_{e^{1}_{2}}$ are 
		not acyclic this chain complex is not totally regular. 
		
		If we change the order of the basis elements 
		$e^{1}_{1}$ and $e^{1}_{2}$, 
		we get an ordering which is not regular. 

		\smallskip
		\item  We consider an example for a chain complex cone
		again 
		(\siehe{Examples~\ref{exs:acyclexam}\,(\ref{exacyc4})}). 
		Let $(C,\Omega)$ be a 
		finite chain complex of order $2$ over $\Z$. Let 
		$\Omega_{2}=\{e^{2}_{1},e^{2}_{2},e^{2}_{3}\}$, 
		$\Omega_{1}=\{e^{1}_{1},e^{1}_{2},e^{1}_{3},e^{1}_{4}\}$,
		$\Omega_{0}=\{e^{0}_{1},e^{0}_{2}\}$ and: 
		\begin{align*}
			\bm_{2}(e^{2}_{1}) & = e^{1}_{1}+e^{1}_{2},
			&\bm_{1}(e^{1}_{1}) & = e^{0}_{2}-e^{0}_{1},  \\
			\bm_{2}(e^{2}_{2}) & = e^{1}_{2}+e^{1}_{3}, 
			&\bm_{1}(e^{1}_{2}) & = e^{0}_{1}-e^{0}_{2},  \\
			\bm_{2}(e^{2}_{3}) & = e^{1}_{3}+e^{1}_{4},
			&\bm_{1}(e^{1}_{3}) & = e^{0}_{2}-e^{0}_{1}, \\
			&
			&\bm_{1}(e^{1}_{4}) & = e^{0}_{1}-e^{0}_{2}.
		\end{align*}
		This chain complex is shellable and regular. As every 
		subcomplex $C_{e^{\nu}_{j}}$ is 
		acyclic we conclude that this chain complex is even totally regular. 
		
		The reader may notice that this chain complex does not come 
		from a simplicial complex! 

		\smallskip
		\item  We have a look at another example for a chain complex cone
		(\siehe{Examples~\ref{exs:acyclexam}\,(\ref{exacyc3})}). 
		Let $(C,\Omega)$ be a finite chain complex of order $2$ 
		over $\Z$ with bases 
		$\Omega_{2}=\{e^{2}_{1}\}$, 
		$\Omega_{1}=\{e^{1}_{1},e^{1}_{2}\}$ and
		$\Omega_{0}=\{e^{0}_{1},e^{0}_{2}\}$ so that:
		\begin{align*}
			\bm_{2}(e^{2}_{1}) & = e^{1}_{1}+e^{1}_{2}, 
			& \bm_{1}(e^{1}_{1}) & = e^{0}_{2}-e^{0}_{1},  \\
			&
			&\bm_{1}(e^{1}_{2}) & = e^{0}_{1}-e^{0}_{2}.
		\end{align*}		
		As above, this chain complex is shellable, regular and every 
		subcomplex $C_{e^{\nu}_{j}}$ is acyclic. So this chain complex 
		is totally regular, too. It also does not come from a simplicial complex. 
		
		\smallskip
		\item  We modify our last example in some detail.  
		Let $(C,\Omega)$ be a finite chain complex of order $2$ 
		over $\Z$ with bases 
		$\Omega_{2}=\{e^{2}_{1},e^{2}_{2}\}$, 
		$\Omega_{1}=\{e^{1}_{1},e^{1}_{2}\}$ and
		$\Omega_{0}=\{e^{0}_{1},e^{0}_{2}\}$ so that:
		\begin{align*}
			\bm_{2}(e^{2}_{1}) & = e^{1}_{1}+e^{1}_{2}, 
			& \bm_{1}(e^{1}_{1}) & = e^{0}_{2}-e^{0}_{1},  \\
			\bm_{2}(e^{2}_{2}) & = e^{1}_{1}+e^{1}_{2},
			&\bm_{1}(e^{1}_{2}) & = e^{0}_{1}-e^{0}_{2}.
		\end{align*}		
		This chain complex is still totally regular, but it refuses to 
		be acyclic as there is a critical basis element $e^{2}_{2}$. 
	\end{enumerate}
\end{exs}

\begin{rem}
  In Section~\ref{ssec:defshell} we considered some examples of 
  shellable chain complexes of order $1$. The chain complexes in 
  the Examples~\ref{exs:shellexam}\,(\ref{exshell1}) and (\ref{exshell2})
  are all regular, but the complexes in the 
  Example~\ref{exs:shellexam}\,(\ref{exshell3})
  are not. Hence, the homology of regular chain complexes is not clear in general. 
\end{rem}

\subsection{$i$-Skeletons of Regular Chain Complexes}
\label{ssec:regskeleton}
\begin{lemma}
	\label{lem:skeletonreg}
	Let $(C,\Omega)$ be a pure regular chain complex, finite of order 
	$d\geq 1$. The $(d-1)$-skeleton $\sk_{d-1}(C)$ of $C$ is regular, too. 
	If $(C,\Omega)$ is even totally regular, then 
	$\sk_{d-1}(C)$ is also totally regular. 
\end{lemma}

\begin{proof}
	As $C$ is a pure chain complex we have 
	$\Gamma=\Omega_{d}=\{e^{d}_{1},\ldots,e^{d}_{k_{d}}\}$ in a regular order. 
	
	The $(d-1)$-skeleton $\sk_{d-1}(C)$ is a pure subcomplex of $C$ 
	of order $(d-1)$ whose basis is 
	$\Omega\setminus\Omega_{d}$. We have to show that 
	$\widehat\Gamma:=\Omega_{d-1}=\bigcup_{i=1}^{k_{d}}\bd(e^{d}_{i})$ 
	has a regular order. 
	
	From Lemma~\ref{lem:skeletonshell} we know that the subcomplex 
	$\sk_{d-1}(C)$ is shellable. 
	Furthermore, $\sk_{d-1}(C)$ 
	satisfies the second regularity condition as this property 
	transmits from $C$. So, we have only to check the first regularity 
	condition which we will do by induction. 
	
	\medskip
	At first, we consider the subcomplex of $\sk_{d-1}(C)$ 
	which is generated by 
	$\widehat\Omega_{e^{d}_{1}}:=\Omega_{e^{d}_{1}}\setminus\{e^{d}_{1}\}$. 
	We have $(\widehat\Omega_{e^{d}_{1}})_{d-1}=\bd(e^{d}_{1})$. 
	As $C$ is regular there exists a shelling of $\bd(e^{d}_{1})$ 
	so that the second regularity condition is fulfilled. Hence, 
	the elements in $(\widehat\Omega_{e^{d}_{1}})_{d-1}$ satisfy 
	the first regularity condition, \ie the chosen shelling of 
	$\bd(e^{d}_{1})$ is a regular order. 
	If $k_{d}=1$, we are done by now. 
	
	\medskip
	If $k_{d}>1$, let $1\leq \ell<k_{d}$. We assume: the basis elements 
	in $\bigcup_{i=1}^{\ell}(\Omega_{e^{d}_{i}})_{d-1}$ which may be 
	ordered in a shelling satisfy the first regularity condition (\ie 
	the subcomplex generated by 
	$\bigcup_{i=1}^{\ell}(\Omega_{e^{d}_{i}}\setminus\{e^{d}_{i}\})$ 
	is regular). 
	We want to show that this also holds for 
	$\bigl(\bigcup_{i=1}^{\ell}(\Omega_{e^{d}_{i}})_{d-1}\bigr)
	\cup(\Omega_{e^{d}_{\ell+1}})_{d-1}$. 
	
	If $e^{d}_{\ell+1}$ is precritical, then 
	$\bigl(\bigcup_{i=1}^{\ell}(\Omega_{e^{d}_{i}})_{d-1}\bigr)
	\cup(\Omega_{e^{d}_{\ell+1}})_{d-1}
	=\bigcup_{i=1}^{\ell}(\Omega_{e^{d}_{i}})_{d-1}$, and there is 
	nothing left to do. So we assume that $e^{d}_{\ell+1}$ is not 
	precritical. 	
	Let 
	\begin{align*}
		\bigcup_{i=1}^{\ell}(\Omega_{e^{d}_{i}})_{d-1}
		&=\{e_{1},\ldots,e_{r},f_{1},\ldots,f_{s}\} \\
		\text{and}\quad
		(\Omega_{e^{d}_{\ell+1}})_{d-1}
		&=\{f_{1},\ldots,f_{s},h_{1},\ldots,h_{t}\}. 
	\end{align*}
	We have $t\geq 1$ and, because of shellability, $s\geq 1$. 
	By assumption, the first regularity condition is fulfilled by the 
	elements of $\bigcup_{i=1}^{\ell}(\Omega_{e^{d}_{i}})_{d-1}$,
	so we have only to consider $h_{1},\ldots,h_{t}$. 
	
	Let there be some $1\leq k\leq t$ so that
	\begin{equation*}
		\bd(h_{k})\subset
		\Bigl(\bigcup_{i=1}^{r}\bd(e_{i})\Bigr)
		\cup
		\Bigl(\bigcup_{i=1}^{s}\bd(f_{i})\Bigr)
		\cup
		\underbrace{
		\Bigl(\bigcup_{i=1}^{k-1}\bd(h_{i})\Bigr).
		}_{=\emptyset \text{ if $k=1$}}
	\end{equation*}
	As the chain complex $C$ is shellable we know that
	\begin{equation*}
		\Bigl(\bigcup_{i=1}^{\ell}\Omega_{e^{d}_{i}}\Bigr)
		\cap\Omega_{e^{d}_{\ell+1}}
		=\bigcup_{i=1}^{s}\Omega_{f_{i}}.
	\end{equation*}
	Therefore, 
	$\bd(h_{k})\cap\bigl(\bigcup_{i=1}^{r}\bd(e_{i})\bigr)
	\subset\bigl(\bigcup_{i=1}^{s}\bd(f_{i})\bigr)$, hence:
	\begin{equation*}
		\bd(h_{k})\subset
		\Bigl(\bigcup_{i=1}^{s}\bd(f_{i})\Bigr)
		\cup
		\Bigl(\bigcup_{i=1}^{k-1}\bd(h_{i})\Bigr).
	\end{equation*}
	As the second regularity condition holds for the elements in 
	$(\Omega_{e^{d}_{\ell+1}})_{d-1}$
	there exist elements
	$a_{1},\ldots,a_{s},b_{1},\ldots,b_{k}\in R$ with $b_{k}\neq 0$ so 
	that
	\begin{align*}
		b_{k}\bm_{d-1}(h_{k})
		&=\sum_{i=1}^{s}a_{i}\bm_{d-1}(f_{i})
		+\sum_{j=1}^{k-1}b_{j}\bm_{d-1}(h_{j}) \\
		&=\sum_{i=1}^{r}0\cdot\bm_{d-1}(e_{i})
		+\sum_{i=1}^{s}a_{i}\bm_{d-1}(f_{i})
		+\sum_{j=1}^{k-1}b_{j}\bm_{d-1}(h_{j})
	\end{align*}
	Hence, the first regularity condition is fulfilled, so 
	$\sk_{d-1}(C)$ is a regular chain complex. 
	
	Additionally, every subcomplex $C_{e^{\nu}_{i}}$ of $\sk_{d-1}(C)$
	is acyclic if this holds for $C$. So, if $C$ is totally regular,
	then $\sk_{d-1}(C)$ is totally regular, too.
\end{proof}

\begin{thm}
	\label{thm:skeletonreg}
	Let $(C,\Omega)$ be a regular chain complex, finite of
	order $d\geq 1$. Then the $(d-1)$-skeleton $\sk_{d-1}(C)$ of $C$ is
	also regular. 
	If $(C,\Omega)$ is even totally regular, 
	$\sk_{d-1}(C)$ is also totally regular. 
\end{thm}

\begin{proof}
	Lemma~\ref{lem:skeletonreg} deals with this statement for pure
	regular chain complexes so there is only the non-pure case to do.
	
	It is clear that the $(d-1)$-skeleton $\sk_{d-1}(C)$ of $C$ 
	satisfies the second regularity condition so we have to prove only 
	the first one. 

	Let $\Gamma := 
	\set{e\in\Omega}{e\not\in\bd(f) \text{ for all }f\in\Omega}$
	be the subset of $\Omega$ of all 
	maximal basis elements. 
	Let the elements of $\Gamma$ be ordered in a regular order. 
	
	Let $\Omega_{d}=\{e^{d}_{1},\ldots,e^{d}_{k_{d}}\}\subset\Gamma$, 
	and 
	$\Gamma\cap\Omega_{d-1}=\{g^{d-1}_{1},\ldots,g^{d-1}_{m}\}$. 
	As a regular order is always monotonically descending, in 
	the chosen regular order of $\Gamma$ the elements of 
	$\Omega_{d}$ come first, followed by all elements of 
	$\Gamma\cap\Omega_{d-1}$. 
	The basis of the chain module 
	$\bigl(\sk_{d-1}(C)\bigr)_{d-1}$ is
	\begin{equation*}
		\Omega_{d-1}=\Bigl(\bigcup_{i=1}^{k_{d}}\bd(e^{d}_{i})\Bigr)
		\cup\{g^{d-1}_{1},\ldots,g^{d-1}_{m}\}.
	\end{equation*}
	The subcomplex $\widehat C$ of $C$ with basis 
	$\bigcup_{i=1}^{k_{d}}\Omega_{e^{d}_{i}}$ is pure and regular. 
	The basis of its chain module $\widehat C_{d-1}$ is 
	$\widehat\Omega_{d-1}=\bigcup_{i=1}^{k_{d}}\bd(e^{d}_{i})$. 
	
	By Lemma~\ref{lem:skeletonreg}, the elements in $\widehat\Omega_{d-1}$ 
	satisfy the first regularity condition. The same holds for 
	$\{g^{d-1}_{1},\ldots,g^{d-1}_{m}\}$ because $C$ is regular. Hence, 
	the $(d-1)$-skeleton $\sk_{d-1}(C)$ of $C$ is regular. 
	
	If $C$ is totally regular, then any subcomplex $C_{e^{\nu}_{i}}$ of
	$\sk_{d-1}(C)\subset C$ is acyclic, so $\sk_{d-1}(C)$ is totally 
	regular. 
\end{proof}

The next two corollaries follow directly from 
Theorem~\ref{thm:skeletonreg}. 

\begin{cor}
	Let $(C,\Omega)$ be a regular chain complex, finite of order $d$.
	For $0\leq i\leq d$, every $i$-skeleton $\sk_{i}(C)$ of $C$ is
	regular.
\end{cor}

\begin{cor}
  \label{cor:totallyregularskeleton}
	Let $(C,\Omega)$ be a totally regular chain complex, finite of order $d$.
	For $0\leq i\leq d$, every $i$-skeleton $\sk_{i}(C)$ of $C$ is
	totally regular.
\end{cor}

\subsection{Homology of Pure Totally Regular Chain Complexes}
\label{ssec:homologypure}
We give a description of the homology of totally regular chain complexes.
We will start with pure complexes and 
compute the homology for a special case. But first we need 
some facts about reduced homology of totally regular chain complexes. 

\begin{lemma}
  Let $(C,\Omega)$ be a totally regular chain complex of order $d$. 
  There is no element $x\in C_{1}$ so that $\anzahl{\bd(x)=1}$. 
\end{lemma}

\begin{proof}
  Let $\Omega_{1}:=\{e^{1}_{1},\ldots,e^{1}_{k_{1}}\}$
  and $\Omega_{0}:=\{e^{0}_{1},\ldots,e^{0}_{k_{0}}\}$ 
  be the bases of the chain modules $C_{1}$ resp. $C_{0}$.
  Because $C$ is totally regular, these orderings of $\Omega_{1}$ 
  and $\Omega_{0}$ are shellings and every subcomplex $C_{e^{1}_{i}}$ 
  is acyclic for $1\leq i\leq k_{1}$. Then we know by 
  Lemma~\ref{lem:anzahlcnull} that $\anzahl{\bd(e^{1}_{i})=2}$ 
  for $1\leq i\leq k_{1}$. 

  We assume that an element $x\in C_{1}$ exists so that 
  $\anzahl{\bd(x)=1}$. 
  Let $i_{0}:=\max\set{1\leq i\leq k_{1}}{a_{i}\neq 0}$, so we get:
  \begin{equation*}
    x=\sum_{i=1}^{i_{0}-1}a_{i}e^{1}_{i} + a_{i_{0}}e^{1}_{i_{0}}.
  \end{equation*}
  As the $1$-skeleton of $C$ is shellable  
  $\bd(e^{1}_{i_{0}})\cap \bigl(\bigcup_{i=1}^{i_{0}-1}\bd(e^{1}_{i})\bigr)\neq\emptyset$. 
  So we distinguish two cases:
  \begin{itemize}
    \item $\bd(e^{1}_{i_{0}})\subset \bigl(\bigcup_{i=1}^{i_{0}-1}\bd(e^{1}_{i})\bigr)$.
      Because the $1$-skeleton of C is also totally regular due to 
      Lemma~\ref{cor:totallyregularskeleton} there are $\lambda_{i}$ for 
      $1\leq i\leq i_{0}$, $\lambda_{i_{0}}\neq 0$, so that
      \begin{equation*}
	\lambda_{i_{0}}\bm_{1}(e^{1}_{i_{0}}) = \sum_{i=1}^{i_{0}-1}\lambda_{i}\bm_{1}(e^{1}_{i}).
      \end{equation*}
      Therefore, we get:
      \begin{align*}
	\bm_{1}(\lambda_{i_{0}}x) &= 
	  \sum_{i=1}^{i_{0}-1}a_{i}\lambda_{i_{0}}\bm_{1}(e^{1}_{i}) + a_{i_{0}}\lambda_{i_{0}}\bm_{1}(e^{1}_{i_{0}}) \\
	&= \sum_{i=1}^{i_{0}-1}\bigl(a_{i}\lambda_{i_{0}} + a_{i_{0}}\lambda_{i}\bigr)\bm_{1}(e^{1}_{i}),
      \end{align*}
      and $\anzahl\bd(\lambda_{i_{0}}x)=1$ because of $\bm_{1}(\lambda_{i_{0}}x) = \lambda_{i_{0}}\bm_{1}(x)$.

    \item $\anzahl{\bigl(\bd(e^{1}_{i_{0}})\cap \bigl(\bigcup_{i=1}^{i_{0}-1}\bd(e^{1}_{i})\bigr)\bigr)}=1$.
       Then $\bd(x) = \bd(e^{1}_{i_{0}}) \setminus \bd\bigl(\sum_{i=1}^{i_{0}-1}a_{i}e^{1}_{i}\bigr)$.
      So we get $\anzahl{\bd\bigl(\sum_{i=1}^{i_{0}-1}a_{i}e^{1}_{i}\bigr)} =1$. 
  \end{itemize}
  In both cases, we get an element of $C_{1}$ which is a linear combination of $e^{1}_{1},\ldots,e^{1}_{i_{0}-1}$
  having only one element in its boundary. Iterating this way delivers a contradiction as 
  $\anzahl{\bd(e^{1}_{1})}=2$. 
\end{proof}

By Lemma~\ref{lem:epsilon} there is a $R$-linear mapping 
$\epsilon\colon C_{0}\rightarrow R$ with $\epsilon\circ\bm_{1}=0$
and $\epsilon(e^{0}_{i})\neq 0$ for all $e^{0}_{i}\in\Omega_{0}$ 
for any totally regular chain complex $(C,\Omega)$. 
Hence, for totally regular chain complexes reduced and nonreduced 
homology are different: 
$H_{0}(C)\cong\widetilde H_{0}(C)\oplus R$. 
We need this fact in the proof of the following theorem. 

\begin{thm}
	\label{thm:hompurenoncrit}
	Let $(C,\Omega)$ be a pure totally regular chain complex of 
	order~$d$. Let there be only noncritical basis elements in 
	$\Omega_{d}$. Then, the chain complex $(C,\Omega)$ is acyclic. 
\end{thm}

\begin{rem}
	The chain complex $(C,\Omega)$ above is pure, totally regular without 
	precritical basis elements. If 
	$\Omega_{d}=\{e^{d}_{1},\ldots,e^{d}_{k_{d}}\}$ is ordered in a regular order, 
	the first regularity condition delivers the following for such a chain complex:
	For any 
	$e^{d}_{\ell}\in\Omega_{d}$,
	$2\leq\ell\leq k_{d}$,
	there exists an element $e^{d-1}_{j_{\ell}}\in\bd(e^{d}_{\ell})$ 
	so that
	\begin{equation*}
		e^{d-1}_{j_{\ell}}\not\in\bigcup_{i=1}^{\ell-1}\bd(e^{d}_{i}).
	\end{equation*}
\end{rem}

\begin{proof}[Proof of Theorem~\ref{thm:hompurenoncrit}]
	We use induction to the order $d$. 
	
	\medskip
	If $d=0$, then $(C,\Omega)$ is a chain complex of order $0$ without 
	precritical basis elements in $\Omega_{0}$. Hence, 
	$\Omega_{0}=\{e^{0}_{1}\}$ and $C_{0}=\erz{e^{0}_{1}}$. Therefore, 
	$H_{0}(C)\cong R$, \ie the chain complex $C$ is acyclic. 
	
	\medskip
	Let $d\geq 1$. We assume that the theorem's statement is true for 
	pure totally regular chain complexes of order $(d-1)$ without 
	precritical basis elements. 
	
	Let $\Omega_{d}:=\{e^{d}_{1},\ldots,e^{d}_{k_{d}}\}$, so that its 
	basis elements are ordered in a regular order. For 
	$1\leq \ell\leq k_{d}$ we define:
	\begin{equation*}
		\Phi_{\ell}:=\bigcup_{i=1}^{\ell}\Omega_{e^{d}_{i}}.
	\end{equation*}
	We get
	$\Phi_{\ell-1}\cap\Omega_{e^{d}_{\ell}}		
	=\bigl(\bigcup_{i=1}^{\ell-1}\Omega_{e^{d}_{i}}\bigr)\cap\Omega_{e^{d}_{\ell}}$
	for $2\leq \ell\leq k_{d}$.
	About the chain complex $P_{\ell}\subset C$ having 
	$\Phi_{\ell-1}\cap\Omega_{e^{d}_{\ell}}$ as basis we know the 
	following:
	\begin{enumerate}
		\item  $P_{\ell}$ is a pure chain complex of order $(d-1)$ and 
		shellable because $C$ is shellable. 
	
		\item  Every subcomplex $C_{e^{\nu}_{i}}$ of $P_{\ell}$ is 
		shellable and acyclic. 
	
		\item  For every 
		$e^{\nu}_{i}\in(\Phi_{\ell-1}\cap\Omega_{e^{d}_{\ell}})$ holds:
		$\Omega_{e^{\nu}_{i}}\subset(\Phi_{\ell-1}\cap\Omega_{e^{d}_{\ell}})$. 
		Hence, the second regularity condition holds for $P_{\ell}$. 
	
		\item $P_{\ell}$ is a subcomplex of $C_{e^{d}_{\ell}}$. 
		As $e^{d}_{\ell}$ fulfils the second regularity condition and 
		$(\Phi_{\ell-1}\cap\Omega_{e^{d}_{\ell}})_{d-1}
		=\bigl(\bigl(\bigcup_{i=1}^{\ell-1}\Omega_{e^{d}_{i}}\bigr)\cap\Omega_{e^{d}_{\ell}}\bigr)_{d-1}
		\subset(\Omega_{e^{d}_{\ell}})_{d-1}
		$
		we conclude that 
		$P_{\ell}$ satisfies the first regularity condition. 
	\end{enumerate}
	Hence $P_{\ell}$ is a pure totally regular chain complex of order 
	$(d-1)$. If $P_{\ell}$ has no precritical elements in 
	$(\Phi_{\ell-1}\cap\Omega_{e^{d}_{\ell}})_{d-1}$, we can apply our 
	induction hypothesis, \ie $P_{\ell}$ is acyclic then. 
	
	So we have to show that $P_{\ell}$ has no precritical elements in 
	$(\Phi_{\ell-1}\cap\Omega_{e^{d}_{\ell}})_{d-1}$. We will do this 
	separately for $d=1$ and for $d\geq 2$. 
	\begin{description}
		\item[$d=1$]  We consider the chain complex 
		$C_{1}\rightarrow C_{0}\rightarrow 0$ with 
		$C_{1}=\erz{e^{1}_{1},\ldots,e^{1}_{k_{1}}}$, $k_{1}\geq 1$,
		about which we know the following:
		\begin{enumerate}
			\item  $\Omega_{1}=\{e^{1}_{1},\ldots,e^{1}_{k_{1}}\}$ 
			has no precriticical basis elements.
		
			\item  Every subcomplex $C_{e^{1}_{i}}$ is acyclic by
			definition. Due to Lemma~\ref{lem:anzahlcnull} we know that
			each chain module $(C_{e^{1}_{i}})_{0}$ is generated by exactly
			two elements, \ie $\anzahl(\Omega_{e^{1}_{i}})_{0}=2$.
		
			\item  $\Phi_{\ell-1}\cap\Omega_{e^{1}_{\ell}}
			=\bigl(\bigcup_{i=1}^{\ell-1}\Omega_{e^{1}_{i}}\bigr)
			\cap\Omega_{e^{1}_{\ell}}
			\neq\emptyset$
			for $2\leq \ell\leq k_{1}$ as the chain complex $C$ is 
			shellable. 
		\end{enumerate}
		So we get 
		$\anzahl(\Phi_{\ell-1}\cap\Omega_{e^{1}_{\ell}})
		=\anzahl(\Phi_{\ell-1}\cap\Omega_{e^{1}_{\ell}})_{0}
		\geq 1$. 
		Furthermore, we know that 
		\begin{equation*}
			\Phi_{\ell-1}\cap\Omega_{e^{1}_{\ell}}
			=\Biggl(
			\Bigl(\bigcup_{i=1}^{\ell-1}\Omega_{e^{1}_{i}}\Bigr)
			\cap\Omega_{e^{1}_{\ell}}
			\Biggr)_{0}
			\subset(\Omega_{e^{1}_{\ell}})_{0}.
		\end{equation*}
		As $\anzahl(\Omega_{e^{1}_{i}})_{0}=2$ we conclude:
		$\anzahl(\Phi_{\ell-1}\cap\Omega_{e^{1}_{\ell}})_{0}\leq 2$.
		If $\anzahl(\Phi_{\ell-1}\cap\Omega_{e^{1}_{\ell}})_{0}=2$,
		then $e^{1}_{\ell}$ would be precritical which contradicts our 
		assumption. 
		
		Therefore 
		$\anzahl(\Phi_{\ell-1}\cap\Omega_{e^{1}_{\ell}})_{0}=1$ for $2\leq 
		\ell\leq k_{1}$, so $\Phi_{\ell-1}\cap\Omega_{e^{1}_{\ell}}$ 
		has no precritical basis elements. 
	
		\item[$d\geq2$]  For $2\leq\ell\leq k_{d}$ let 
		$(\Phi_{\ell-1}\cap\Omega_{e^{d}_{\ell}})_{d-1}
		=\{g_{1},\ldots,g_{m_{\ell}}\}$, ordered in a regular order.
		If $m_{\ell}=1$, there is nothing left to do, so let 
		$m_{\ell}\geq 2$. 
		
		We assume: There is some $2\leq j\leq m_{\ell}$ so that $g_{j}$ 
		is precritical:
		\begin{equation*}
			a_{j}\bm_{d-1}g_{j}=\sum_{i=1}^{j-1}a_{i}\bm_{d-1}g_{i}
			\quad\text{with $a_{i}\in R$ for $1\leq i\leq j$ and $a_{j}\neq 0$. }
		\end{equation*}
		So there holds: 
		$a_{j}g_{j}-\sum\limits_{i=1}^{j-1}a_{i}g_{i}\in\ker\bm_{d-1}$. 
		
		By assumption, the chain complex $C$ has no precritical 
		elements in $\Omega_{d}$. 
		Therefore,
		$(\Phi_{\ell-1}\cap\Omega_{e^{d}_{\ell}})_{d-1}
		\subsetneqq(\Omega_{e^{d}_{\ell}})_{d-1}$. 
		Otherwise, we would get:
		\begin{align*}
			(\Omega_{e^{d}_{\ell}})_{d-1}
			=\bd(e^{d}_{\ell})
			\subset(\Phi_{\ell-1})_{d-1}
			&=\Bigl(\bigcup_{i=1}^{\ell-1}\Omega_{e^{d}_{i}}\Bigr)_{d-1} 
			\\
			&=\bigcup_{i=1}^{\ell-1}(\Omega_{e^{d}_{i}})_{d-1}
			=\bigcup_{i=1}^{\ell-1}\bd({e^{d}_{i}}),
		\end{align*}
		and because of the first regularity condition 
		$e^{d}_{\ell}$ would be precritical then. \contradiction
		
		Hence,  
		$\bm_{d}(e^{d}_{\ell})=\sum\limits_{i=1}^{m_{\ell}}b_{i}g_{i}+r$ with
		$\sum\limits_{i=1}^{m_{\ell}}b_{i}g_{i}
			\in\erz{\Phi_{\ell-1}\cap\Omega_{e^{d}_{\ell}}}$
		and some
		$r\in\erz{(\Omega_{e^{d}_{\ell}})_{d-1}\setminus\Phi_{\ell-1}}$.
		As $(\Omega_{e^{d}_{\ell}})_{d-1}=\bd({e^{d}_{\ell}})$ it is 
		$r\neq 0$. 
		
		Furthermore, we know $\bm_{d-1}\circ\bm_{d}(e^{d}_{\ell})=0$, 
		\ie $\bm_{d}(e^{d}_{\ell})\in\ker\bm_{d-1}$. 
		So, as $r\neq 0$, there are two linearly independent elements 
		in $\ker\bm_{d-1}$, namely $\bm_{d}(e^{d}_{\ell})$ and 
		$a_{j}g_{j}-\sum_{i=1}^{j-1}a_{i}g_{i}$. 
		These both elements are even contained in 
		$(C_{e^{d}_{\ell}})_{d-1}$, 
		so we conclude:
		$\ker\bm_{d-1}|_{C_{e^{d}_{\ell}}}$ is generated by at least 
		two elements. 
		
		We know that $(C_{e^{d}_{\ell}})_{d}=\erz{e^{d}_{\ell}}$, 
		therefore $\im\bm_{d}|_{C_{e^{d}_{\ell}}}$ is generated by one 
		element. Hence we conclude:
		$H_{d-1}(C_{e^{d}_{\ell}})\neq 0$ \contradiction
		 -- a contradiction to our assumption that 
		$C_{e^{d}_{\ell}}$ is acyclic, \ie $H_{d-1}(C_{e^{d}_{\ell}})=0$ if 
		$d\geq 2$. 	
		
		Therefore, $(\Phi_{\ell-1}\cap\Omega_{e^{d}_{\ell}})_{d-1}$ 
		contains no precritical basis elements. 
	\end{description}
	By induction hypothesis we conclude for $2\leq \ell\leq k_{d}$: The
	chain complex $P_{\ell}$ with basis
	$(\Phi_{\ell-1}\cap\Omega_{e^{d}_{\ell}})$ is acyclic.
	
	\medskip
	Let $Q_{j}$ be the chain complex with basis 
	$\Phi_{j}=\bigcup_{i=1}^{j}\Omega_{e^{d}_{i}}$ for 
	$1\leq j\leq k_{d}$. 
	Because $\Omega_{d}$ is ordered in a regular order, $Q_{j}$ is 
	a totally regular chain complex. 
	In particular, $Q_{1}=\erz{\Omega_{e^{d}_{1}}}=C_{e^{d}_{1}}$ 
	and $Q_{k_{d}}=\erz{\bigcup_{i=1}^{k_{d}}\Omega_{e^{d}_{i}}}=C$. 
	By assumption, $Q_{1}$ is an acyclic chain complex. 
	
	As 
	$(Q_{\ell})_{n}
	  =\erz{(\Phi_{\ell-1})_{n}\cup(\Omega_{e^{d}_{\ell}})_{n}}
	  =\erz{(\Phi_{\ell-1})_{n}}+\erz{(\Omega_{e^{d}_{\ell}})_{n}}
	  =(Q_{\ell-1})_{n}+(C_{e^{d}_{\ell}})_{n}$
	and
	$(P_{\ell})_{n}=(Q_{\ell-1})_{n}\cap(C_{e^{d}_{\ell}})_{n}$
	for any $2\leq \ell \leq k_{d}$ and
	for any $n\geq 0$
	we get the following exakt sequence:
	\begin{equation*}
	  0 \rightarrow
	  (P_{\ell})_{n} \stackrel{\varphi_{n}}{\longrightarrow}
	  (Q_{\ell-1})_{n}\oplus (C_{e^{d}_{\ell}})_{n} \stackrel{\psi_{n}}{\longrightarrow}
	  (Q_{\ell})_{n} \rightarrow 0 
	\end{equation*}
	with $\varphi_{n}(x) = (x,-x)$, $\psi_{n}(x,y) = x+y$ \cite[cf.][page~149]{Hatcher.2008}.
	Because $(C,\Omega)$ is a totally regular chain complex there is 
	some mapping $\epsilon\colon C_{0}\rightarrow R$ with 
	$\epsilon(e^{0}_{i})\neq 0$ for all $e^{0}_{i}\in\Omega_{0}$. 
	Hence each chain complex $P_{\ell}$, $Q_{\ell-1}$, $C_{e^{d}_{\ell}}$
	and $Q_{\ell}$ can be augmented by the restriction of $\epsilon$ 
	\cite[cf.][page~150]{Hatcher.2008}. 

	Furthermore, $\varphi_{n}\circ\bm_{n+1}=\bm_{n}\circ\varphi_{n+1}$ and 
	$\psi_{n}\circ\bm_{n+1}=\bm_{n}\circ\psi_{n+1}$. 
	As in \citet[page~116]{Hatcher.2008} we obtain a long exact sequence
	of reduced homology groups:
	\begin{align*}
	  \ldots \rightarrow
	  \widetilde H_{m}(P_{\ell}) \stackrel{\varphi_{\ast n}}{\longrightarrow}
	  \widetilde H_{m}(Q_{\ell-1}) \oplus \widetilde H_{m}(C_{e^{d}_{\ell}}) \stackrel{\psi_{\ast n}}{\longrightarrow}
	  \widetilde H_{m}(Q_{\ell}) \stackrel{\delta_{n}}{\longrightarrow}
	  \widetilde H_{m-1}(P_{\ell}) \stackrel{\varphi_{\ast n-1}}{\longrightarrow} \ldots\\
	  \ldots \stackrel{\delta_{1}}{\longrightarrow}
	  \widetilde H_{0}(P_{\ell}) \stackrel{\varphi_{\ast 0}}{\longrightarrow}
	  \widetilde H_{0}(Q_{\ell-1}) \oplus \widetilde H_{0}(C_{e^{d}_{\ell}}) \stackrel{\psi_{\ast 0}}{\longrightarrow}
	  \widetilde H_{0}(Q_{\ell}) \stackrel{\delta_{0}}{\longrightarrow}
	  0.  
	\end{align*}
	For an acyclic chain complex all reduced homology groups are $0$. As 
	$P_{\ell}$, $Q_{\ell-1}$ and $C_{e^{d}_{\ell}}$ are all acyclic chain complexes
	we conclude that $Q_{\ell}$ is acyclic, too. 

	Therefore, $C=Q_{k_{d}}$ is an acyclic chain complex, \ie 
	$H_{i}(C)=0$ for $i\geq 1$ and $H_{0}(C)\cong R$. 
\end{proof}

\begin{thm}
	\label{thm:hompure}
	Let $(C,\Omega)$ be a pure totally regular chain complex of 
	order~$d\geq 1$. Let the basis 
	$\Omega_{d}=\{e^{d}_{1},\ldots,e^{d}_{k_{d}}\}$ of $C_{d}$ have 
	$n<k_{d}$ precritical elements. Then there holds for the homology 
	of $C$:
	\begin{align*}
		H_{d}(C) & \cong R^{n};  \\
		H_{i}(C) & = 0 \quad\text{for $i\neq 0, d$;} \\
		H_{0}(C) & \cong R.
	\end{align*}
\end{thm}

\begin{rem}
  $H_{0}(C)\cong R^{k_{0}}$ if $d=0$, cf. Remark~\ref{rem:complexoforderzero}.
\end{rem}

\begin{proof}
	We can assume that all noncritical elements in $\Omega_{d}$ come 
	first in the regular order. Otherwise we can change the order so that 
	the precritical (and critical) elements of $\Omega_{d}$ come at 
	last; this has no influence to the shellability and regularity of 
	$C$. Let $m:=k_{d}-n$, then we have:
	\begin{equation*}
		\Omega_{d}=\{
		\underbrace{e^{d}_{1},\ldots,e^{d}_{m}}_{\text{noncritical}},
		\underbrace{e^{d}_{m+1},\ldots,e^{d}_{k_{d}}}_{\text{precritical}}\}.
	\end{equation*}
	Due to Theorem~\ref{thm:precritical} we know that 
	$H_{d}(C)\cong R^{n}$. 
	
	Consider now the chain complex 
	$\widehat C:=\erz{\bigcup_{i=1}^{m}\Omega_{e^{d}_{i}}}$ which is a 
	pure subcomplex of $C$ of order $d$. Its chain modules are 
	$\widehat C_{d}=\erz{e^{d}_{1},\ldots,e^{d}_{m}}$ and 
	$\widehat C_{\nu}=C_{\nu}$ for $0\leq\nu\leq d-1$. 
	As $C$ is a totally regular chain complex this holds for 
	$\widehat C$, too. In opposite to $C$ the chain complex 
	$\widehat C$ has no precritical elements. Due to 
	Theorem~\ref{thm:hompurenoncrit} the complex $\widehat C$ is acyclic, \ie 
	$H_{0}(\widehat C)\cong R$ and 
	$H_{i}(\widehat C)=0$ for $i\geq 1$. 
	
	Because $\widehat C_{d-1}=C_{d-1}$ and $H_{d-1}(\widehat C)=0$
	we get
	\begin{equation*}
		\im\bm_{d}|_{\widehat C_{d}}
		\subset\im\bm_{d}
		\subset\ker\bm_{d-1}
		=\ker\bm_{d-1}|_{\widehat C_{d-1}}
		=\im\bm_{d}|_{\widehat C_{d}}.
	\end{equation*}
	Therefore, $\im\bm_{d}|_{\widehat C_{d}}=\im\bm_{d}$.
	As $\widehat C_{\nu}=C_{\nu}$ for $0\leq\nu\leq d-1$ we conclude:
	\begin{align*}
		H_{i}(C) & =H_{i}(\widehat C)=0 
		\quad\text{for $1\leq i\leq d-1$,} \\
		H_{0}(C) & =H_{0}(\widehat C) \cong R.
		\qedhere
	\end{align*}
\end{proof}

\subsection{Homology of arbitrary Totally Regular Chain Complexes}
\label{ssec:homologyarb}
We consider the general case and compute the homology of 
arbitrary totally regular chain complexes. 

\begin{thm}
	Let $(C,\Omega)$ be a totally regular chain complex of 
	order~$d\geq 1$. For any $0\leq\nu\leq d$ let 
	$\Omega_{\nu}:=\{e^{\nu}_{1},\ldots,e^{\nu}_{k_{\nu}}\}$ be the 
	basis of the chain module $C_{\nu}$. 
	Let 
	$\Gamma$ be the subset of $\Omega$ which contains all maximal basis elements. 
	For $0\leq\nu\leq d$, let there be $n_{\nu} < k_{\nu}$ precritical 
	elements in $(\Gamma\cap\Omega_{\nu})$. 
	Then the homology of $C$ is:
	\begin{align*}
		H_{i}(C) & \cong R^{n_{i}} \quad\text{for $1\leq i\leq d$;} \\
		H_{0}(C) & \cong R^{n_{0}+1}.
	\end{align*}
\end{thm}

\begin{proof}
	We assume that $\Gamma$ is ordered in a regular order.
	Then each $\Omega_{\nu}$ is ordered so that the elements of 
	$\Gamma$ come at last; let
	\begin{equation*}
		\Omega_{\nu}:=\{
		\underbrace{e^{\nu}_{1},\ldots,e^{\nu}_{m_{\nu}}}_{\not\in\Gamma},
		\underbrace{e^{\nu}_{m_{\nu}+1},\ldots,e^{\nu}_{k_{\nu}}}_{\in\Gamma}\}.
	\end{equation*}
	So, we even have $n_{\nu}\leq k_{\nu}-m_{\nu}$.
	
	We consider the chain complex 
	$C^{d}:=\erz{\bigcup_{i=1}^{k_{d}}\Omega_{e^{d}_{i}}}$ which is a 
	subcomplex of $C$. It is pure of order $d$ and totally regular. 
	The reader may notice that $C^{d}=C$ if $\Gamma=\Omega_{d}$. 
	According to Theorem~\ref{thm:hompure} we know:
	\begin{align*}
		H_{d}(C^{d}) & \cong R^{n_{d}};  \\
		H_{i}(C^{d}) & = 0 \quad\text{for $i\neq 0, d$;} \\
		H_{0}(C^{d}) & \cong R.
	\end{align*}
	As $(C^{d})_{d}=C_{d}$ we get $H_{d}(C) \cong R^{n_{d}}$.
	
	For $1\leq\mu\leq d-1$ 
	we consider the chain complex
	$C^{\mu}:=\erz{\bigcup_{i=1}^{k_{\mu}}\Omega_{e^{\mu}_{i}}}
	\subset\sk_{\mu}(C)$ which is pure and finite of order $\mu$.  
	We know that each $\mu$-skeleton of $C$ is shellable 
	and totally regular. 
	Let $\Omega_{\sk_{\mu}}$ be the basis of $\sk_{\mu}(C)$ and 
	$\Gamma_{\sk_{\mu}}:=
	\set{e\in\Omega_{\sk_{\mu}}}{e\not\in\bd(f) 
	\text{ for all }f\in\Omega_{\sk_{\mu}}}$, ordered in a 
	regular order. 
	As a regular order is always monotonically descending, 
	the elements $e^{\mu}_{1},\ldots,e^{\mu}_{k_{\mu}}$ come first 
	in the regular order of $\Gamma_{\sk_{\mu}}$. Therefore, the complex 
	$C^{\mu}$ is also shellable and totally regular. 
	
	Let there be $\ell_{\mu}$ precritical basis elements in 
	$\{e^{\mu}_{1},\ldots,e^{\mu}_{m_{\mu}}\}$. 
	By Theorem~\ref{thm:hompure} we get:
	\begin{align*}
		H_{\mu}(C^{\mu}) & \cong R^{n_{\mu}+\ell_{\mu}};  \\
		H_{i}(C^{\mu}) & = 0 \quad\text{for $i\neq 0, \mu$;} \\
		H_{0}(C^{\mu}) & \cong R.
	\end{align*}
	Therefore, we also know $H_{\mu}(C^{\mu+1})=0$. Because
	$(C^{\mu+1})_{\mu+1}=C_{\mu+1}$ we get: 
	\begin{equation*}
		\im\bm_{\mu+1}=\im\bm_{\mu+1}|_{C^{\mu+1}}
		=\ker\bm_{\mu}|_{\erz{e^{\mu}_{1},\ldots,e^{\mu}_{m_{\mu}}}}.
	\end{equation*}
	As the chain complex 
	$\widehat C^{\mu}:=\erz{\bigcup_{i=1}^{m_{\mu}}\Omega_{e^{\mu}_{i}}}
	\subset C^{\mu}$ is pure of order $\mu$ and totally regular, there 
	holds due to Theorem~\ref{thm:hompure}:
	$H_{\mu}(\widehat C^{\mu})
	\cong\ker\bm_{\mu}|_{\erz{e^{\mu}_{1},\ldots,e^{\mu}_{m_{\mu}}}}
	\cong R^{\ell_{\mu}}$. 
	
	It is $\Gamma\cap\Omega_{\mu}
	=\{e^{\mu}_{m_{\mu}+1},\ldots,e^{\mu}_{k_{\mu}}\}$, 
	hence 
	$e^{\mu}_{i}\not\in\im\bm_{\mu+1}$ 
	for $m_{\mu}+1\leq i\leq k_{\mu}$. 
	We conclude:
	\begin{equation*}
		H_{\mu}(C)=\quotient{\ker\bm_{\mu}}{\im\bm_{\mu+1}}
		\cong R^{n_{\mu}}
		\quad\text{for $1\leq\mu\leq d-1$}. 
	\end{equation*}

	For $\mu=0$ we consider the chain complex 
	$C^{1}:=\erz{\bigcup_{i=1}^{k_{1}}\Omega_{e^{1}_{i}}}\subset\sk_{1}(C)$ 
	which is pure of order $1$ and totally regular. 
	Hence, $H_{0}(C^{1})
	=\quotient{\erz{e^{0}_{1},\ldots,e^{0}_{m_{0}}}}{\im\bm_{1}}
	\cong R$. 
	Because $e^{0}_{i}\not\in\im\bm_{1}$ for $m_{0}+1\leq i\leq k_{0}$
	we get:	
	\begin{equation*}
		H_{0}(C)=\quotient{C_{0}}{\im\bm_{1}}
		\cong R^{n_{0}+1}.
		\qedhere
	\end{equation*}
\end{proof}

\section*{Acknowledgements}
\noindent
We would like to thank Eva-Maria Feichtner and Dmitry N. Kozlov 
for bringing chain complexes to our attention. Many thanks also 
to Ralf Donau who had a good idea how to prove the existence of 
monotonically descending shellings and to Emanuele Delucchi for some good 
advice and helpful discussions.

\end{document}